\documentclass[11pt]{article}
\usepackage{latexsym, amsmath, amssymb, amsthm}
\usepackage{enumerate, mathrsfs, bbm}
\usepackage{graphicx, tikz}
\usepackage{authblk, setspace, cite}
\usepackage[top = 1 in, bottom = 1 in, left = 1.2 in, right = 1.2 in]{geometry}

\newtheorem{theorem}{Theorem}[section]
\newtheorem{lemma}[theorem]{Lemma}
\newtheorem{corollary}[theorem]{Corollary}
\newtheorem{proposition}[theorem]{Proposition}

\newtheorem{definition}[theorem]{Definition}

\newtheorem{claim}[theorem]{Claim}

\numberwithin{equation}{section}

\makeatletter
\def\blfootnote{\xdef\@thefnmark{}\@footnotetext}
\makeatother
 
\def\m{\mathbb}		\def\mcal{\mathcal}			
\def\lam{\lambda}		\def\Lam{\Lambda}
\def\eps{\epsilon}  	\def\th{\theta}		\def\vp{\varphi}
\def\a{\alpha}	\def\b{\beta}	\def\ls{\lesssim}
\def\p{\partial}  
\def\wh{\widehat}	\def\wt{\widetilde}		
\def\la{\langle}	\def\ra{\rangle}
\def\ls{\lesssim}	\def\gs{\gtrsim}
		
\def\be{\begin{equation}}     \def\ee{\end{equation}}
\def\bp{\begin{pmatrix}}	\def\ep{\end{pmatrix}} 
		\def\F{\mathscr{F}}
\def\R{\mathbb{R}}

\newcommand{\dd}{\mathrm{d}}

\begin{document}
\title{Global well-posedness of the Majda-Biello system in the resonant case on the real line}
%\title{Global well-posedness for the Majda-Biello systems}
\author[]{Xin Yang}
\date{}
\maketitle

\begin{abstract}
We study the Cauchy problem for the following Majda-Biello system in the case $\alpha=4$,
where the resonance effect is the most significant, on the real line.
\[
\left\{
\begin{array}{rcl}
	u_{t} + u_{xxx} & = & - v v_x,\\
	v_{t} + \alpha v_{xxx} & = & - (uv)_{x},\\
	(u,v)|_{t=0} & = & (u_0,v_0) \in H^{s}(\mathbb{R}) \times H^{s}(\mathbb{R}),
\end{array}
\right. \quad x \in \mathbb{R}, \, t \in \mathbb{R}.
\]
For Sobolev regularity $s\in[\frac34, 1)$, we establish global well-posedness by refining the I-method. Previously, the critical index for local well-posedness was known to be $\frac34$, while global well-posedness was only obtained for $s\geq 1$. 
Our global well-posedness result bridges the gap and matches the threshold in the local theory. The main novelty of our approach is to introduce a pair of distinct $I$-operators, tailored to the resonant structure of the Majda-Biello system with $\alpha=4$. This dual-operator framework allows for pointwise control of the multipliers in the modified energies constructed via the multilinear correction technique. These modified energies are almost conserved and provide effective control over the Sobolev norm of the solution for all time. This new approach has potential applications to other coupled dispersive systems exhibiting strong resonant interactions.

%We study the global well-posedness for the initial value problem (IVP) of Majda-Biello systems:
%\[
%\left\{
%\begin{array}{l}
%	u_{t} + u_{xxx} = - v v_x,\\
%	v_{t} + \alpha v_{xxx} = - (uv)_{x},\\
%	(u,v)|_{t=0} = (u_0,v_0) \in H^{s}(\mathbb{R}) \times H^{s}(\mathbb{R}),
%\end{array}
%\right. \quad x \in \mathbb{R}, \, t \in \mathbb{R},
%\]
%where $\alpha \in \mathbb{R}\setminus \{0\}$. Let $s^{*}(\alpha)$  be the smallest value for which the IVP  is locally analytically well-posed in $H^{s}(\mathbb{R})\times H^{s}(\mathbb{R}) $ when $s > s^{*}(\alpha)$. So far, all values of $s^{*}(\alpha)$ have been determined and most local well-posedness results have been extended to the global ones except when $\alpha = 4$ in which the resonance is the most significant and $s^{*}(4)$ is found to be $\frac34$ while the global well-posedness is only known for $s\geq 1$. 
%
%In this work, we apply a refined I-method scheme by introducing different $I$-operators for $u$ and $v$. These two $I$-operators are closely connected via the dispersion coefficient $\alpha=4$ and it is this relation that overcomes the difficulty in deriving the pointwise upper bounds of multipliers in the I-method. As a result, we fill the gap to establish the global well-posedness of the Majda-Biello system with $\a=4$ in $H^{s}(\mathbb{R})\times H^{s}(\mathbb{R}) $ for any $s\geq \frac34$.
\end{abstract}

\blfootnote{\hspace{-0.25in} 2010 Mathematics Subject Classification. 35Q53; 35G55; 35L56; 35D30.}

\blfootnote{\hspace{-0.25in} Key words and phrases. KdV-KdV systems;   Majda-Biello systems; Global well-posedness; I-Method}

\begin{center}
\tableofcontents
\end{center}

\section{Introduction}
\subsection{Literature review and main result}
Consider the initial value problem (IVP) of the real-valued Majda-Biello systems on the real line $\m{R}$:
\be\label{MB}
\left\{
\begin{array}{rcl}
	u_{t} + u_{xxx} & = & - v v_x,\\
	v_{t} + \a v_{xxx} & = & - (uv)_{x},\\
	(u,v)|_{t=0} & = & (u_0,v_0) \in \mathscr{H}^{s}(\m{R}),
\end{array}
\right. 
\ee
where $x$ and $t$ are real numbers, $\alpha \in \mathbb{R} \setminus \{0\}$. The function space $\mathscr{H}^{s}(\mathbb{R})$ is defined as the Cartesian product $H^{s}(\mathbb{R}) \times H^{s}(\mathbb{R})$, where $H^{s}(\mathbb{R})$ denotes the standard Sobolev space. 
The system (\ref{MB}) was  proposed by Majda and Biello in \cite{MB03} as a reduced asymptotic model to study the nonlinear resonant interactions of long wavelength equatorial Rossby waves and barotropic Rossby waves. A fundamental question about (\ref{MB}) is its well-posedness. We aim to determine the minimal regularity of the initial data $(u_0, v_0)$ for which (\ref{MB}) is well-posed. 

The Majda-Biello systems (\ref{MB}) belong to a wider class of coupled KdV-KdV systems (see e.g. \cite{YZ22a} for a more detailed discussion), so we first review the well-posedness theory for the single KdV equation on the real line $\m{R}$ or the periodic torus $\m{T}$:
\begin{equation}\label{KdV}
	\left\{\begin{aligned}
		&u_t + u_{xxx} + uu_x = 0,\\
		&u\mid_{t=0} = u_0. 
	\end{aligned}\right.
\end{equation}
where $u_0\in H^s(\mathbb{R})$ or $H^s(\mathbb{T})$. 
Over more than 60 years of development, the Cauchy problem of the KdV equation (\ref{KdV}) has inspired many great works. Established results show that (\ref{KdV}) is globally well-posed for any $s\geq -1$ \cite{KT06, KV19}, and this index represents the optimal regularity threshold \cite{Mol11, Mol12}. In contrast, if the solution map is required to be uniformly continuous (which is more demanding than the continuity requirement in the standard Hadamard's well-posedness definition), then the critical index is $-3/4$ for the case on $\m{R}$ and $-1/2$ for the case on $\m{T}$ \cite{Bou93a, KPV96, CKSTT03, CCT03, Guo09, Kis09}.

Unlike the single KdV equation, most coupled KdV-KdV systems, such as the Majda-Biello systems (\ref{MB}), are not completely integrable, which makes the study of these systems more challenging. 
%When $a_2\in(0,4]\backslash\{1\}$, the situation is more complicated. Oh\cite{Oh09} proved   analytical LWP of (\ref{M-B system}) in $H_{0}^{s}(\m{T})\times H^{s} (\m{T})$ for any $s\geq \min\{1,\,s^{*}(a_2)+\}$, where $s^{*}(a_2)$ is some constant only depending on $a_2$. The index $s^{*}(a_2)$ is no less than $\frac12$ and equals to $\frac12$ for almost every $a_{2}\in(0,4]\backslash\{1\}$. As a result, for almost every $a_2\in(0,4]\backslash\{1\}$, (\ref{M-B system}) is locally analytically well-posed in $H_{0}^{s}(\m{T})\times H^{s} (\m{T})$ for $s>\frac12$. It is worth mentioning that the regularity threshold for the LWP is sharp in the sense that if $s<\min\{1,\,s^{*}(a_2)\}$, then (\ref{M-B system}) is not locally $C^{3}$ well-posed in $H_{0}^{s}(\m{T})\times H^{s}(\m{T})$. Based on the LWP result, Oh\cite{Oh09-2} further established a similar result for the GWP. More specifically, for almost every $a_2\in(0,4)\backslash\{1\}$, (\ref{M-B system}) is globally well-posed in $H_{0}^{s}(\m{T})\times H^{s}(\m{T})$ for $s>\frac57$. 
For the results of (\ref{MB}) on the torus $\m{T}$, we refer the readers to \cite{Oh09, Oh09-2, YZ22b} for detailed discussions. Next, we will briefly review previous well-posedness results of (\ref{MB}) on the real line $\m{R}$. 

\textbf{Case 1: $\a = 1$}. This situation is similar to the single KdV case and the local well-posedness of (\ref{MB}) in $\mcal{H}^{s}(\m{R})$ for any $s>-\frac{3}{4}$ follows immediately from the single KdV theory. The corresponding global well-posedness result was established by Oh\cite{Oh09-2} via the I-method, where the asymmetry in some multipliers needs to be taken care of to obtain pointwise estimates. The asymmetry is caused by the structure of the Majda-Biello system which was not present in the single KdV case.
	
\textbf{Case 2: $\a \in (0,4) \setminus\{1\}$}. This situation is quite different from the single KdV case due to certain resonance effect. As a result, the threshold for the critical index of the initial data regularity needs to be higher. Oh in \cite{Oh09} proved that (\ref{MB}) is locally well-posed  in $\mcal{H}^{s}(\m{R})$ for $s\geq 0$ and  ill-posed when $s<0$ if the solution map is required to be $C^{2}$. Moreover, due to the conservation of the mass $M$ of the local solution:
	\[
		M(u,v)(t) := \int_{\R} \left[ u^2(x,t) + v^2(x,t) \right] \,dx,
	\]
the global well-posedness of (\ref{MB}) in $\mcal{H}^{s}(\m{R})$ for $s\geq 0$ automatically holds.
	
\textbf{Case 3: $\a = 4$}. In this case, the resonance effect of the Majda-Biello system is the most significant, so the requirement for the regularity of the initial data is the highest. 	Yang-Zhang\cite{YZ22a} showed that (\ref{MB}) is locally well-posed  in $\mcal{H}^{s}(\m{R})$ for $s\geq \frac34$. This index $\frac34$ is also sharp if the solution map is required to be $C^2$, see \cite{YLZ24}. When $s\geq 1$, by taking advantage of the conservation of the above mass $M$ and the following energy $E$ of the local solution:
	\begin{align*}
		E(u,v)(t) := \int_{\R} \left[ u_x^2(x,t) + \a v_x^2(x,t) - u(x,t) v^2(x,t) \right] \, dx,
	\end{align*}
the global well-posedness of (\ref{MB}) can also be justified. Thus, the remaining challenge is the global well-posedness when $s\in [\frac34, 1)$.	
	
\textbf{Case 4: $\a < 0$ or $\a > 4$}. In this case, the resonance effect is not present and the treatment of the well-posedness theory is similar to that when $\a = 1$. It turns out that the critical index for the well-posedness is also $-\frac34$ when the solution map is required to be $C^3$, see e.g. \cite{Oh09, YZ22a, YLZ24}. 
%Moreover, the GWP can also be justified using the standard I-method. 

Based on the above literature review, the remaining global well-posedness issue of (\ref{MB}) is when $\a=4$ and $s\in[\frac34, 1)$, so our goal in this paper is to fill this gap. The following is the main result of this paper. 

\begin{theorem}\label{Thm, main}
For the Majda-Biello system (\ref{MB}) with $\a=4$, it is globally well-posed in the space $\mathscr{H}^{s}(\m{R})$ for any $s\in [\frac{3}{4}, 1)$.
\end{theorem}

Besides Majda-Biello systems, two other widely studied coupled KdV-KdV systems are the Hirota-Satsuma systems (\ref{H-S system}):
\be\label{H-S system}
\left\{\begin{array}{rll}
	u_{t}+ a_{1}u_{xxx} &=& -6a_{1}uu_{x}+c_{12}vv_{x},   \vspace{0.03in}\\
	v_{t}+v_{xxx} &=& -3uv_{x}, \vspace{0.03in}\\
	\left. (u,v)\right |_{t=0} &=& (u_{0},v_{0}),
\end{array}\right.\ee
where $a_{1}, c_{12}\in\m{R}\setminus\{0\}$, and the Gear-Grimshaw systems (\ref{G-G system}):
\be\label{G-G system}
\left\{\begin{array}{rcl}
	u_{t}+u_{xxx}+\sigma_{3}v_{xxx} &=&-uu_{x}+\sigma_{1}vv_{x}+\sigma_{2}(uv)_{x},\vspace{0.03in}\\
	\rho_{1}v_{t}+\rho_{2}\sigma_{3}u_{xxx}+v_{xxx}+\sigma_{4}v_{x} &=& \rho_{2}\sigma_{2}uu_{x}-vv_{x}+\rho_{2}\sigma_{1}(uv)_{x}, \vspace{0.03in}\\
	\left. (u,v)\right |_{t=0} &=& (u_{0},v_{0}),
\end{array}\right.\ee
where $\sigma_{i}\in\m{R} \, (1\leq i\leq 4)$ and $\rho_{1},\,\rho_{2}>0$. The Hirota-Satsuma systems were proposed in \cite{HS81} to describe the interaction of two long waves with different dispersion relations. Their well-posedness problem was studied in e.g. \cite{HS81, AC08, Fen94, YZ22a, YZ22b}. The Gear-Grimshaw systems were derived in \cite{GG84} (also see \cite{BPST92} for the explanation about the physical context) as a model to describe the strong interaction of two-dimensional internal gravity waves propagating in a stratified fluid. Their well-posedness theory can be found in e.g.  \cite{GG84, AC08, ACW96, BPST92, LP04, ST00, YZ22a, YZ22b}. Nonetheless, due to the complexity of the structure and interaction terms, the global well-posedness theory of (\ref{H-S system}) and (\ref{G-G system}) are far from complete. For example, as shown in Theorem 1.5 and Theorem 1.6 in \cite{YZ22a}, the following two cases are similar to that for the Majda-Biello system (\ref{MB}) with $\a=4$ in that the resonance effect is the most significant and the critical index is $\frac34$.
\begin{itemize}
\item Case 1: The Hirota-Satsuma system (\ref{H-S system}) with $a_1 = \frac14$.

\item Case 2: The Gear-Grimshaw system (\ref{G-G system}) with 
\[\rho_2\sigma_3^2\leq \frac{9}{25}, \quad  \rho_{1}^2+\frac{25\rho_2\sigma_3^2-17}{4}\rho_1+1=0.\]
%$\rho_2\sigma_3^2\leq \frac{9}{25}$ and $\rho_{1}^2+\frac{25\rho_2\sigma_3^2-17}{4}\rho_1+1=0$.
\end{itemize}
We believe the method developed in this paper can handle the above two cases as well. 

\subsection{Outline of the approach}

Next, we would like to illustrate this paper's method. When $\a=4$, the Majda-Biello system is specified to be (\ref{MB-4}). There are two well-known conserved quantities $M(u,v)$ and $E(u,v)$, see (\ref{eq_M}) and (\ref{eq_E}), which are at the $L^2$ level and $H^1$ level respectively. When $s\in[\frac34, 1)$, we first apply the idea of the I-method \cite{CKSTT03}, to modify $u,v$ at the high-frequency region to improve their regularity to $H^{1}$ level. 
According to (\ref{MB-4}), the dispersion coefficient for $u$ is $1$, and for $v$ is $\a=4$ which is different from $1$, so our new input is to introduce two different $I$-operators $I_1$ and $I_2$, which are closely related via $\a$, such that both $I_1 u$ and $I_2 v$ belong to $H^{1}(\m{R})$. More precisely, by denoting $m_1$ and $m_2$ to be the Fourier multipliers of $I_1$ and $I_2$ respectively, where $m_1$ is the standard $I$-operator, see (\ref{m1}), while $m_2$ is chosen as a scaled version of $m_1$:
\[m_2(\xi) := m_1(\rho_* \xi).\]
The optimal choice of the constant $\rho_*$ turns out to be $2$.
%The constant $\rho_*$ is determined by the following relation:
%\be\label{rho_star}
%	\rho_{*}^{2-2s} = \a \big( |r|^{2s+1} + |1+r|^{2s+1} \big),
%\ee
%where $r$ is any root of the quadratic function $f$ defined by 
%\be\label{char_fn}
%	f(x) := 3\a x^2 + 3\a x + (\a-1), \quad\forall\,x\in\m{R}.
%\ee
%In particular, when $\a=4$, $f$ has two repeated roots $r=-\frac12$, which implies that $\rho_* = 2$ for any $s\in[\frac34, 1)$. 

Now we will explain why $\rho_*$ is chosen as $2$. Firstly, since $I_1 u$ and $I_2 v$ belong to $H^1(\m{R})$, it is a standard strategy to consider the first modified energy $E^{(1)}$ defined by
$E^{(1)}(u,v) := E(I_1 u, I_2 v)$.
Unfortunately, $E^{(1)}(u,v)$ is not conserved anymore and its growing rate is not controllable either, so we apply the multilinear correction technique in \cite{CKSTT03} (also see \cite{Oh09-2}) to introduce the second modified energy $E^{(2)}(u,v)$ which contains a free multiplier $\sigma_3$. The hope is to find some suitable $\sigma_3$ such that the growing rate of $E^{(2)}(u,v)$ can be controlled. By direct computation, the time derivative of $E^{(2)}(u,v)$ is determined by three multipliers (see (\ref{me_2nd_td})): $M_{31}$ which is a tri-linear multiplier, and $M_{41}$ and $M_{42}$ which are two quadri-linear multipliers. Since the operators with quadri-linear multipliers usually decay faster, the hope is that $M_{31}$ vanishes. As a result, $\sigma_3$ needs to be chosen as 
\be\label{sigma_3-choice}
	\sigma_3 = \sigma_3(\xi_1, \xi_2,\xi_3) = \frac{\xi_1^3 m_{1}^{2}(\xi_1) + 4\xi_2^3 m_{2}^{2}(\xi_2) + 4\xi_3^3 m_{2}^{2}(\xi_3)}{\xi_1^3 + 4\xi_2^3 + 4\xi_3^3}, \quad\forall\, \sum_{k=1}^{3} \xi_k = 0.
\ee
Meanwhile, since the main parts of both $M_{41}$ and $M_{42}$ are also in terms of $\sigma_3$, see (\ref{mpl_4}), it is crucial to bound $\sigma_3$ pointwisely in order to estimate $M_{41}$ and $M_{42}$. Noticing that the denominator of $\sigma_3$ can be factored to be:
\[ \xi_1^3 + 4\xi_2^3 + 4\xi_3^3 = 3 \xi_1 (2\xi_2 + \xi_1) (2\xi_3 + \xi_1), \quad\forall\, \sum_{k=1}^{3} \xi_k = 0,\]
Therefore, for the sake of bounding $\sigma_3$, its numerator needs to vanish when $\xi_2 = -\frac12 \xi_1$ or $\xi_3 = -\frac12 \xi_1$. By setting $\xi_2 = -\frac12 \xi_1$ or $\xi_3 = -\frac12 \xi_1$, and solving 
\[
	\xi_1^3 m_{1}^{2}(\xi_1) + 4\xi_2^3 m_{2}^{2}(\xi_2) + 4\xi_3^3 m_{2}^{2}(\xi_3) = 0, 
\]
we find $\rho_*$ has to be $2$, see Section \ref{Subsec, choice_I2} for more general analysis. With this choice of $\rho_{*}$, we manage to obtain effective pointwise bounds on $\sigma_3$, $M_{41}$ and $M_{42}$. Thanks to these pointwise bounds, the rest argument for justifying global well-posedness is standard as that in \cite{CKSTT03}.

As a summary, the main novelty of this paper lies in its refined I-method framework, particularly the use of two coupled $I$-operators, to achieve the global well-posedness at the critical regularity $s=\frac34$ for a previously unresolved case $(\a=4)$ in the Majda-Biello systems (\ref{MB}). 

\subsection{Organization of the paper}
The organization of this paper is as follows: In Section \ref{Sec, me}, we introduce two coupled $I$-operators $I_1$ and $I_2$, and their related modified energies $E^{(1)}(u,v)$ and $E^{(2)}(u,v)$ whose time derivatives are computed in the frequency space. Based on these time derivatives, we determine the optimal choice of the coupled $I$-operators in Section \ref{Sec, est_mpl}, and derive pointwise bounds for all multilinear multipliers in the formula of $\p_t E^{(2)}(u,v)$. Thanks to these pointwise bounds and standard properties of local solutions, we are able to control the growth rate of $E^{(2)}(u,v)$ in Section \ref{Sec, est_me}. Meanwhile, we also bound the difference between $E^{(1)}$ and $E^{(2)}$ for any fixed time point. Finally, in Section \ref{Sec, gwp}, by applying a scaling argument and an iteration scheme, we manage to control the growth of $E^{(1)}(u,v)$ which dominates the $\mathcal{H}^{s}$ norm of $(u,v)$. Thus, the global well-posedness result in Theorem \ref{Thm, main} is established.

\section{Modified energies}
\label{Sec, me}
\subsection{The first modified energy}
We study the IVP of the Majda-Biello system (\ref{MB}) with $\a=4$ and $\frac34\leq s < 1$:
\be\label{MB-4}
\left\{
\begin{array}{rcl}
	u_{t} + u_{xxx} & = & - v v_x,\\
	v_{t} + 4 v_{xxx} & = & - (uv)_{x},\\
	(u,v)|_{t=0} & = & (u_0,v_0) \in \mathscr{H}^{s}(\m{R}).
\end{array}
\right. 
\ee
It has been known that (\ref{MB-4}) is locally well-posed in $\mathscr{H}^{s}(\m{R})$ for any $s\geq \frac34$, now the question is whether the local well-posedness can be strengthened to be the global well-posedness? Usually, one needs to rely on conservation laws to achieve this goal. If the local solution $(u,v)$ has $C\big([0,T]; \mathscr{H}^1(\m{R}) \big)$ regularity for some $T>0$, then the following two quantities are conserved over the time range $[0,T]$.
\begin{align}
	M(u,v)(t) &:= \int_{\R} \left[ u^2(x,t) + v^2(x,t) \right] \,dx, \label{eq_M} \\
	E(u,v)(t) &:= \int_{\R} \left[ u_x^2(x,t) + 4 v_x^2(x,t) - u(x,t) v^2(x,t) \right] \, dx. \label{eq_E}
\end{align}
The mass $M$ is at the $L^2$ level and the energy $E$ is at the $H^1$ level. Although the last term $uv^2$ in the integrand in (\ref{eq_E}) has indefinite sign, it is easy to figure out that the conservation of both $E$ and $M$ guarantees the boundedness of the $\mathcal{H}^{1}$ norm of $(u,v)$, see e.g. Lemma \ref{Lemma, E and H1} in Section \ref{Sec, energy_relation}.
However, when $s\in[\frac34, 1)$, the regularity of $(u,v)$ is not enough to utilize energy $E$ and there is no conservation law available in $H^{s}$ level, so we need to come up with other ways to control the $\mathscr{H}^{s}(\m{R})$ norm of the local solution $(u,v)$.

In order to do so, we adopt the I-method which was first systematically introduced in \cite{CKSTT03}. The key idea is to modify the local solution $(u,v)$ so that their regularity is improved to be at $\mathscr{H}^{1}$ level to apply the energy $E$. Since there are two unknown functions $u$ and $v$ which correspond to different dispersion coefficients, we may need to modify them differently. We first fix a smooth decreasing even function $\varphi\in C^{\infty}(\m{R})$ such that $0\leq \varphi \leq 1$ and 
\be\label{varphi}
\varphi(\xi) = 
\left\{\begin{array}{ccl}
	1 & \text{if} &  |\xi| \leq 1, \\
	|\xi|^{s-1}  & \text{if} & |\xi| \geq 2.
\end{array}\right.
\ee
Then for $j=1,2$, we define the multiplier operator $I_{j}: H^{s}(\m{R}) \to H^{1}(\m{R})$ such that 
\be\label{fm-Iop}
	\wh{I_j w} (\xi) = m_j(\xi) \wh{w}(\xi), \quad \xi \in \R,
\ee
where $\wh{I_j w}$ and $\wh{w}$ represent the Fourier transforms of $I_j w$ and $w$ on $\m{R}$, and $m_1$ and $m_2$ are defined below:
\[m_1(\xi) = \varphi(N^{-1}\xi), \qquad m_2(\xi) = \varphi(\rho_{*} N^{-1}\xi ), \qquad \forall\,\xi\in\m{R}.\]
Here, $\rho_{*}$ denotes a positive constant which will be determined later as the constant $2$ (see (\ref{m2_fix}) in Section \ref{Subsec, choice_I2}), and $N$ refers to any large positive integer which is at least 10. Thanks to the properties of the function $\varphi$, we know $m_j \in C^{\infty}(\m{R})$ is an even function such that $0\leq m_j \leq 1$ for $j=1,2$. In addition, 
\be\label{m1}
m_1(\xi) = 
\left\{\begin{array}{ccl}
	1 & \text{if} &  |\xi| \leq N, \\
	N^{1-s} |\xi|^{s-1}  & \text{if} & |\xi| \geq 2N,
\end{array}\right.
\ee
and
\be\label{m2}
m_2(\xi) = m_1(\rho_{*} \xi) = 
\left\{\begin{array}{ccl}
	1 & \text{if} &  |\xi| \leq N/\rho_{*}, \\
	N^{1-s} \rho_{*}^{s-1} |\xi|^{s-1} & \text{if} & |\xi| \geq 2N/\rho_{*}.
\end{array}\right. 
\ee
Although $I_1$ and $m_1$ depend on $N$, and $I_2$ and $m_2$ depend on both $N$ and $\rho_{*}$, we simply write them as $I_1$, $I_2$, $m_1$ and $m_2$ for convenience of notation. 

The purpose of introducing $I_1$ and $I_2$ is to improve the regularity of $H^s$ functions to be $H^{1}$ so that the energy $E$ of $(I_1 u, I_2 v)$ is validated. Consequently, the first modified energy $E^{(1)}$ is then defined as follows.
\be\label{mE_1st}
	E^{(1)}(u,v) := E(I_1 u, I_2 v), \quad u, v \in H^s(\m{R}).
\ee
Since the local solution $(u,v)$ is real-valued and $\{m_j\}_{j=1,2}$ are real-valued even functions, then we know both $\{I_j u\}_{j=1,2}$ and $\{I_j v\}_{j=1,2}$ are also real-valued. 

Due to the presence of the operators $I_1$ and $I_2$ in (\ref{mE_1st}), the first modified energy $E^{(1)}(u,v)$ may not conserve anymore. The hope is that its growth can be suitably controlled. In order to exploit the symmetry more effectively, we will rewrite $E^{(1)}$ as an integration over a hyperplane in the frequency space.

First, we introduce some notations. For a multiplier $M_n$ on $\m{R}^{n}$ and for functions $f_1, \dots, f_n$ on $\m{R}$, we define
\be\label{eq_Lambda}
	\Lambda_n(M_n; f_1, \dots, f_n) = \int_{\Gamma_n} M_n(\xi_1, \dots, \xi_n) \prod_{j=1}^n \wh{f_j}(\xi_j) \, d\sigma, 
\ee
where $\Gamma_n$ denotes the hyperplane: $\Gamma_n = \{ (\xi_1, \dots, \xi_n) \in \R^n : \sum_{j=1}^n \xi_j = 0 \}$, $\wh{f_j}$ refers to the Fourier transform of $f_j$, and $d\sigma$ means the surface integral. If the functions $\{f_{j}\}_{j=1}^{n}$ are identical, then we adopt the simplified notation: 
\[
	\Lambda_n(M_n; f) := \Lambda_n(M_n; f, \dots, f).
\] 
Based on the above notations, for real-valued functions $u$ and $v$, one can apply the Plancherel identity to  find that 
\begin{align*}
	\int_{\m{R}} u_x^2 \, dx = - \Lambda_2(\xi_1 \xi_2; u), \quad \int_{\m{R}} v_x^2 \, dx = - \Lambda_2(\xi_1 \xi_2; v), \quad 
	\int_{\m{R}} u v^2 \,dx = \Lambda_3(1; u, v, v).
\end{align*}
Thus, the standard energy $E$ in (\ref{eq_E}) can be written as:
\begin{equation}
	E(u,v) = - \Lambda_2(\xi_1 \xi_2; u) - 4 \Lambda_2(\xi_1 \xi_2; v) - \Lambda_3(1; u, v, v). \label{eq_E_freq}
\end{equation}
Likewise, the first modified energy $E^{(1)}$ in (\ref{mE_1st}) expands to:
\be\label{mE_1st_freq}
\begin{split}
	E^{(1)}(u,v) &= - \Lambda_2(\xi_1 \xi_2 m_1(\xi_1) m_1(\xi_2); u) - 4 \Lambda_2(\xi_1 \xi_2 m_2(\xi_1) m_2(\xi_2); v) \\
	& \quad - \Lambda_3(m_1(\xi_1) m_2(\xi_2) m_2(\xi_3); u, v, v).
\end{split}
\ee

\subsection{Time derivative of energies}
In order to control the growth rate of the first modified energy, we need to study how fast it grows by computing its derivative with respect to time.

Firstly, the system (\ref{MB-4}) can be rewritten in the frequency variable:
\be\label{mb-freq}
\left\{\begin{aligned}
	\wh{u}_t(\xi) &= i \xi^3 \wh{u}(\xi) - \frac{1}{2} i \xi \int_{ \eta_1 + \eta_2 = \xi} \wh{v}(\eta_1) \wh{v}(\eta_2) \,d\sigma, \\
	\wh{v}_t(\xi) &= 4 i \xi^3 \wh{v}(\xi) - i \xi \int_{\eta_1 + \eta_2 = \xi} \wh{u}(\eta_1) \wh{v}(\eta_2) \,d\sigma.
\end{aligned}\right.
\ee
For any multiplier $M_2$ on $\m{R}^2$, it follows from (\ref{eq_Lambda}) that 
\[
\p_{t} \Lam_2(M_2; u) = \int_{\Gamma_{2}} M_{2}(\xi_1, \xi_2) \big[ \wh{u}_t(\xi_1) \wh{u}(\xi_2) +  \wh{u}(\xi_1) \wh{u}_t(\xi_2)  \big] \, d\sigma.	
\]
Plugging (\ref{mb-freq}) into the above identity and using symmetry yields 
\[ 
	\p_{t} \Lam_2(M_2; u) = \frac12 i \Lam_3\big(\xi_1\big[ M_2(\xi_1, \xi_2+\xi_3) + M_2(\xi_2+\xi_3, \xi_1) \big]; u, v, v \big).
\]
Since $\sum\limits_{j=1}^{3} \xi_j = 0$, by replacing $\xi_1$ by $-(\xi_2+\xi_3)$ and using symmetry, we obtain 
\[ 
	\p_{t} \Lam_2(M_2; u) = - i \Lam_3\big(\xi_2\big[ M_2(\xi_1, \xi_2+\xi_3) + M_2(\xi_2+\xi_3, \xi_1) \big]; u, v, v \big).
\]
In particular, 
\begin{align*}
	\p_{t} \Lam_2\big(\xi_1\xi_2 m_{1}(\xi_1) m_{1}(\xi_2) ; u \big) &= -i\Lam_3\big( 2\xi_2 \big[ \xi_1(\xi_2+\xi_3)m_1(\xi_1)m_1(\xi_2+\xi_3) \big]; u, v, v \big).
\end{align*}
Since $\xi_2+\xi_3 = -\xi_1$ and $m_1$ is an even function, then the above equation yields
\begin{align*}
	\p_{t} \Lam_2\big(\xi_1\xi_2 m_{1}(\xi_1) m_{1}(\xi_2) ; u \big) &= i\Lam_3\big( 2\xi_2 \xi_1^2 m_1^2(\xi_1); u, v, v \big).
\end{align*}
Thanks to the symmetry, we can replace $2\xi_2$ by $\xi_2 + \xi_3$ (equivalently $-\xi_1$) in the above equation to have
\begin{align*}
	\p_{t} \Lam_2\big(\xi_1\xi_2 m_{1}(\xi_1) m_{1}(\xi_2) ; u \big) &= -i\Lam_3\big( \xi_1^3 m_1^2(\xi_1); u, v, v \big).
\end{align*}
Similarly, we deduce from the second equation in (\ref{mb-freq}) to find
\[
	\p_{t} \Lam_2(M_2; v) = - i \Lam_3\big( (\xi_1+\xi_2) \big[ M_2(\xi_1+\xi_2, \xi_3) + M_2(\xi_3, \xi_1+\xi_2) \big]; u, v, v \big). 
\]
As a consequence, 
\[
	\p_{t} \Lam_2 \big(\xi_1\xi_2 m_{2}(\xi_1) m_{2}(\xi_2); v \big) = -i \Lam_{3} \big( 2(\xi_1+\xi_2) (\xi_1+\xi_2)\xi_3 m_2(\xi_1+\xi_2) m_2(\xi_3); u, v, v\big).
\]
Again, due to the facts that $\xi_1+\xi_2 = -\xi_3 $ and $m_2$ is an even function, the above equation leads to 
\[
	\p_{t} \Lam_2 \big(\xi_1\xi_2 m_{2}(\xi_1) m_{2}(\xi_2); v \big) = -i \Lam_{3} \big( 2\xi_3^3 m_2^2(\xi_3); u, v, v\big).
\]
Since $\xi_2$ and $\xi_3$ are symmetric on the right hand side of the above equation, we conclude that 
\[
	\p_{t} \Lam_2 \big(\xi_1\xi_2 m_{2}(\xi_1) m_{2}(\xi_2); v \big) = -i \Lam_{3} \big( \xi_2^3 m_2^2(\xi_2) + \xi_3^3 m_2^2(\xi_3); u, v, v\big).
\]
In summary, we have obtained the time derivatives of $\Lam_2$ operators as shown below:
\be\label{Lam_2_td}\left\{\begin{array}{rll}
	\p_{t} \Lam_2\big(\xi_1\xi_2 m_{1}(\xi_1) m_{1}(\xi_2) ; u \big) &=& - i \Lam_3\big(\xi_1^3 m_{1}^2(\xi_1); u, v, v \big), \vspace{0.1in} \\
	\p_{t} \Lam_2 \big(\xi_1\xi_2 m_{2}(\xi_1) m_{2}(\xi_2); v \big) &=& - i \Lam_3\big( \xi_2^3 m_{2}^2(\xi_2) + \xi_3^3 m_{2}^2(\xi_3); u, v, v \big). 
\end{array}\right.\ee

Analogously, for any multiplier $M_3$ on $\m{R}^3$, we can justify equation (\ref{gLam_3_td}) below (verification details can be found in Appendix \ref{Sec, td_gLam3}).
\be\label{gLam_3_td}\begin{split}
	\p_{t} \Lam_3(M_3; u, v, v) = &\ i \Lam_3(\a_3 M_3; u,v,v) - \frac{1}{2} i \Lam_{4}\big( [\xi_1+\xi_4] M_3(\xi_1+\xi_4, \xi_2, \xi_3); v \big) \\
	& - i \Lam_{4}\big( [\xi_2+\xi_3] M_3(\xi_1, \xi_2+\xi_3, \xi_4); u,u,v,v \big) \\
	& - i \Lam_{4}\big( [\xi_2+\xi_3] M_3(\xi_1, \xi_4, \xi_2+\xi_3); u,u,v,v \big),
\end{split}\ee
where $\a_3$ is defined as 
\be\label{alpha_3}
	\a_3(\xi_1,\xi_2,\xi_3) := \xi_1^3 + 4\xi_2^3 + 4\xi_3^3, \quad \forall\, (\xi_1,\xi_2,\xi_3)\in \Gamma_3,
\ee
which can be rewritten in a compact form: $\a_3(\xi_1,\xi_2,\xi_3) = -3\xi_1 (\xi_2 - \xi_3)^2$ since $\sum\limits_{i=1}^{3} \xi_i = 0$ for any $(\xi_1,\xi_2,\xi_3)\in \Gamma_3$. 
In particular, when 
\[M_{3}(\xi_1, \xi_2, \xi_3) = m_{1}(\xi_1) m_{2}(\xi_2) m_{2}(\xi_3), \quad\forall\, (\xi_1, \xi_2, \xi_3)\in\Gamma_{3},\]
we obtain 
\be\label{Lam_3_td}\begin{split}
	&\ \p_t \Lambda_3 \big(m_{1}(\xi_1) m_{2}(\xi_2) m_{2}(\xi_3); u, v, v\big) \\
	= &\ i \Lambda_3 ( \alpha_3 m_{1}(\xi_1) m_{2}(\xi_2) m_{2}(\xi_3); u, v, v ) \\
	&- \frac{1}{2} i \Lambda_4 \big( [\xi_1 + \xi_4] m_1(\xi_1 + \xi_4) m_2(\xi_2) m_2(\xi_3); v \big) \\
	&- 2i \Lambda_4 \big( [\xi_2 + \xi_3] m_1(\xi_1) m_2(\xi_2 + \xi_3) m_2(\xi_4); u, u, v, v \big).
\end{split}\ee
Based on (\ref{Lam_2_td}) and (\ref{Lam_3_td}), it follows from (\ref{mE_1st_freq}) that 
\be\label{me_1st_td}\begin{split}
	\p_t E^{(1)}(u,v) = &\ i \Lam_{3}\big( \xi_1^3 m_{1}^{2}(\xi_1) + 4\xi_2^3 m_{2}^{2}(\xi_2) + 4\xi_3^3 m_{2}^{2}(\xi_3); u, v, v \big) \\	
	& - i \Lambda_3 ( \alpha_3 m_{1}(\xi_1) m_{2}(\xi_2) m_{2}(\xi_3); u, v, v ) \\
	& + \frac{1}{2} i \Lambda_4 \big( [\xi_1 + \xi_4] m_1(\xi_1 + \xi_4) m_2(\xi_2) m_2(\xi_3); v \big) \\
	& + 2i \Lambda_4 \big( [\xi_2 + \xi_3] m_1(\xi_1) m_2(\xi_2 + \xi_3) m_2(\xi_4); u, u, v, v \big),
\end{split}\ee
where $\a_3$ is defined as in (\ref{alpha_3}).

\subsection{The second modified energy}

By observing the first two rows in (\ref{me_1st_td}), the combined multiplier
\be\label{main_E1_dt}
	\xi_1^3 m_{1}^{2}(\xi_1) + 4\xi_2^3 m_{2}^{2}(\xi_2) + 4\xi_3^3 m_{2}^{2}(\xi_3) - \alpha_3 m_{1}(\xi_1) m_{2}(\xi_2) m_{2}(\xi_3)
\ee
in the operator $\Lam_3$ is unbounded since the cubic term $m_{1}(\xi_1) m_{2}(\xi_2) m_{2}(\xi_3)$ decays faster than quadratic terms $m_{1}^{2}(\xi_1)$, $m_{2}^{2}(\xi_2)$ and $m_{2}^{2}(\xi_3)$ when $|\xi_1|\sim |\xi_2| \sim |\xi_3| \gg 1$. As a result, the growth of $E^{(1)}(u,v)$ is difficult to control by using (\ref{me_1st_td}). We will apply the multilinear correction technique to modify $E^{(1)}(u,v)$ further. This technique was first proposed in \cite{CKSTT03} to study the single KdV equation and later applied to Majda-Biello systems in \cite{Oh09-2}. More precisely, we give more freedom to the $\Lam_3$ term in (\ref{mE_1st_freq}) to define the second modified energy $E^{2}(u,v)$ as follows:
\be\label{mE_2nd_freq}\begin{split}
E^{(2)}(u,v) &= - \Lambda_2(\xi_1 \xi_2 m_1(\xi_1) m_1(\xi_2); u) - 4 \Lambda_2(\xi_1 \xi_2 m_2(\xi_1) m_2(\xi_2); v)\\
& \quad - \Lambda_3(\sigma_3; u, v, v),
\end{split}\ee
where the multiplier $\sigma_3$ is to be determined. 

The choice of $\sigma_3$ depends on the structure of $\p_{t} E^{(2)}(u,v)$ which will be computed next. According to (\ref{mE_1st_freq}) and (\ref{mE_2nd_freq}), $E^{(2)}(u,v)$ only differs from $E^{(1)}(u,v)$ by a $\Lam_3$ term, so $\p_{t} E^{(2)}(u,v)$ can also be computed similarly. It turns out that 
\be\label{me_2nd_td}
	\p_t E^{(2)}(u,v) = \ i \Lam_{3}\big(M_{31}; u, v, v \big) + i \Lam_4(M_{41}; v) + i \Lam_4(M_{42}; u,u,v,v),
\ee
where 
\be\label{mpl_3}
	M_{31} =  \xi_1^3 m_{1}^{2}(\xi_1) + 4\xi_2^3 m_{2}^{2}(\xi_2) + 4\xi_3^3 m_{2}^{2}(\xi_3) - \a_3\sigma_3, \quad \forall\, (\xi_1,\xi_2,\xi_3)\in \Gamma_3,
\ee
and
\be\label{mpl_4}\left\{\begin{aligned}
	& M_{41} := \frac{1}{2} (\xi_1 + \xi_4) \sigma_3(\xi_1 + \xi_4, \xi_2, \xi_3), \vspace{0.1in}\\
	& M_{42} := (\xi_2 + \xi_3) \big[ \sigma_3(\xi_1, \xi_2 + \xi_3, \xi_4) + \sigma_3(\xi_1, \xi_4, \xi_2 + \xi_3) \big],
\end{aligned}\right.
\quad \forall\, (\xi_1,\xi_2,\xi_3, \xi_4)\in \Gamma_4.
\ee
Since the lower order operator is usually more difficult to control, we hope to get rid of $\Lambda_3$ in (\ref{me_2nd_td}), that is to make $M_{31}=0$. So we choose $\sigma_3$ as 
\be\label{sigma_3}
\sigma_3(\xi_1, \xi_2, \xi_3) = \frac{\xi_1^3 m_1^2(\xi_1) + 4\xi_2^3 m_2^2(\xi_2) + 4\xi_3^3 m_2^2(\xi_3)}{\alpha_3(\xi_1, \xi_2, \xi_3)}, \quad \forall\, (\xi_1,\xi_2,\xi_3)\in \Gamma_3,
\ee
where $\alpha_3(\xi_1, \xi_2, \xi_3)$ is as defined in (\ref{alpha_3}). Consequently,
\be\label{E2_td}
\p_t E^{(2)}(u,v) = i \Lambda_4 \left( M_{41}; v \right) + i \Lambda_4 \left( M_{42}; u, u, v, v \right).
\ee
Noting that  $\a_3(\xi_1, \xi_2, \xi_3)$ is symmetric with respect to its last two variables $\xi_2$ and $\xi_3$, so $M_{42}$ can be simplified as $M_{42} = 2 (\xi_2 + \xi_3) \sigma_3(\xi_1, \xi_2 + \xi_3, \xi_4)$ for any $(\xi_1,\xi_2,\xi_3, \xi_4)\in \Gamma_4$.

%Combining (\ref{mpl_4}) with (\ref{sigma_3}) yields 
%\be\label{M41}
%	M_{41} = \frac{1}{2} (\xi_1 + \xi_4) \frac{ (\xi_1 + \xi_4)^3 m_1^2(\xi_1 + \xi_4) + 4\xi_2^3 m_2^2(\xi_2) + 4\xi_3^3 m_2^2(\xi_3) }{ (\xi_1 + \xi_4)^3 + 4\xi_2^3 + 4\xi_3^3 }.
%\ee
%On the other hand, since $\a_3(\xi_1, \xi_2, \xi_3)$ is symmetric with respect to its last two variables, the same holds true for $\sigma_3(\xi_1, \xi_2, \xi_3)$ as well according to (\ref{sigma_3}). Hence, 
%\be\label{M42}
%M_{42} = 2 (\xi_2 + \xi_3) \frac{ \xi_1^3 m_1^2(\xi_1) + 4 (\xi_2 + \xi_3)^3 m_2^2(\xi_2 + \xi_3) + 4\xi_4^3 m_2^2(\xi_4) }{ \xi_1^3 + 4 (\xi_2 + \xi_3)^3 + 4\xi_4^3 }. 
%\ee

\section{Estimates on the multipliers}
\label{Sec, est_mpl}
\subsection{Choice of the second $I$-operator $I_2$}
\label{Subsec, choice_I2}
Based on (\ref{E2_td}), it reduces to estimate the multipliers $M_{41}$ and $M_{42}$ in order to analyze the growth of the second modified energy $E^{(2)}$. Meanwhile, since both $M_{41}$ and $M_{42}$ are determined by $\sigma_3$ according to (\ref{mpl_4}), the key ingredient will be the estimate on $\sigma_3$. Next, we will first give a pointwise bound on $\sigma_3$ by suitably choosing the constant $\rho_{*}$ in the definition (\ref{m2}) of the second $I$-operator $I_2$.

According to (\ref{sigma_3}) and (\ref{alpha_3}), 
\be\label{sigma_3_cpt}\begin{split}
\sigma_3(\xi_1, \xi_2, \xi_3) &= \frac{\xi_1^3 m_1^2(\xi_1) + 4\xi_2^3 m_2^2(\xi_2) + 4\xi_3^3 m_2^2(\xi_3)}{\xi_1^3 + 4\xi_2^3 + 4\xi_3^3} \\
&= \frac{\xi_1^3 m_1^2(\xi_1) + 4\xi_2^3 m_2^2(\xi_2) + 4\xi_3^3 m_2^2(\xi_3)}{-3\xi_1 (\xi_2 - \xi_3)^2}, \qquad \forall\, (\xi_1,\xi_2,\xi_3)\in \Gamma_3. 
\end{split}\ee
In order to bound $\sigma_3$ pointwisely, the most challenging part is the resonance region, i.e. the region where the denominator of $\sigma_3$ is near $0$. 

Heuristically, the region 
\[\Omega:= \big\{(\xi_1, \xi_2, \xi_3)\in\Gamma_3: \xi_1 \gg N, \ |\xi_2 - \xi_3| \leq \xi_1^{-1/2} \big\},\]
is critical due to the squared denominator in (\ref{sigma_3_cpt}), so $m_2$ has to be suitably chosen so that the numerator of $\sigma_3$ can cancel out the singularity of the denominator. In the following, we denote $N_{j} = |\xi_j|$, $j=1,2,3$. Then in $\Omega$, we have 
\[\xi_1 = N_1 \gg N, \quad \xi_2 = -\frac12 N_1 + O(N_1^{-1/2}), \quad \xi_3 = -\frac12 N_1 + O(N_1^{-1/2}).\]
Hence, the denominator of $\sigma_3$ is bounded by $1$ and $|\xi_i| \gg N$ for $i\in\{1,2,3\}$ in $\Omega$. Then it follows from (\ref{m1}) and (\ref{m2}) that $\xi_1 = N_1$, $m_1(\xi_1) = N^{1-s} N_1^{s-1}$,
\[
	\xi_2 = - \frac12 N_1\big[ 1 + O(N_1^{-3/2}) \big], \quad m_2(\xi_2) = N^{1-s} \rho_{*}^{s-1} \Big(\frac12 N_1\Big)^{s-1}\big[ 1 + O(N_1^{-3/2})  \big]^{s-1}.
\]
The expressions for $\xi_3$ and $m_2(\xi_3)$ are similar to the above. As a result, the numerator of $\sigma_3$ becomes 
\be\label{num_sigma_3}\begin{split}
&\quad\quad \xi_1^3 m_1^2(\xi_1) + 4\xi_2^3 m_2^2(\xi_2) + 4\xi_3^3 m_2^2(\xi_3) \\
& =\,\,\, \Big[ 1 - \Big(\frac{2}{\rho_*}\Big)^{2-2s} \Big] N^{2-2s} N_{1}^{2s+1} + O\Big( N^{2-2s} N_1^{2s-\frac12} \Big). 
\end{split}\ee
From (\ref{num_sigma_3}), it is readily seen that $\rho_{*} := 2$ is the only choice to cancel out the main term, so now we fix the definition of $m_2$ in (\ref{m2}) by choosing $\rho_{*}=2$:
\be\label{m2_fix}
m_2(\xi) = m_1(2 \xi) = 
\left\{\begin{array}{ccl}
	1 & \text{if} &  |\xi| \leq N/2, \\
	N^{1-s} 2^{s-1} |\xi|^{s-1} & \text{if} & |\xi| \geq N.
\end{array}\right.
\ee

\subsection{Properties of $m_1$}
\label{m1_prop}
Recalling the definition of the multiplier $m_1$: 
\[m_1(\xi) = \varphi(N^{-1}\xi), \qquad\forall\,\xi\in\m{R},\] 
where $N\geq 10$ and $\vp$ is a fixed smooth decreasing even function on $\m{R}$ which satisfies (\ref{varphi}) with the range $[0,1]$. We now present three basic properties of $m_1$ which can also be used to study the multiplier $m_2$ since it is a scaling of $m_1$ as defined in (\ref{m2_fix}).

\begin{lemma}\label{Lemma, m1est}
	Let $\frac34\leq s < 1$. Then there exists a constant $C$, which only depends on $s$, such that 
	\be\label{m1est}
		\frac{1}{C} \la \xi \ra^{s-1} \leq m_1(\xi) \leq C N^{1-s} \la \xi \ra^{s-1}, \quad \forall\,\xi\in\m{R}.
	\ee
\end{lemma}
\begin{proof}
	If $|\xi| \leq 2N$, then $m_1(\xi) \leq 1 \leq (3N)^{1-s} \la \xi \ra^{s-1}$. In addition, since $m_1$ is a decreasing function, then $m_1(\xi) \geq m_1(2N) = 2^{s-1} \geq 2^{s-1} \la\xi\ra^{s-1}$. So (\ref{m1est}) is valid for $C = 3^{1-s}$.
	If $|\xi| > 2N$, then $m_1(\xi) = N^{1-s} |\xi|^{s-1}$. As a result, $m_1(\xi) \leq (2N)^{1-s} \la \xi \ra^{s-1}$ and $m_1(\xi) \geq |\xi|^{s-1} \geq \la \xi\ra^{s-1}$. So (\ref{m1est}) also holds for $C = 3^{1-s}$.
\end{proof}

\begin{lemma}\label{Lemma, m1p1}
	Let $s\in [\frac34, 1)$. Then there exists a positive constant $C$, which only depends on $s$, such that 
	\be\label{m1_db}
		|m_1'(\xi)| \leq C \, \frac{m_1(\xi)}{|\xi|}, \qquad |m_1''(\xi)| \leq C \, \frac{m_1(\xi)}{|\xi|^2}, \qquad\forall\,\xi\neq 0.
	\ee
\end{lemma}
\begin{proof}
	Since $m_1(\xi) = \varphi(N^{-1}\xi)$, then 
	\[m_1'(\xi) = N^{-1} \vp'(N^{-1}\xi), \qquad m_1''(\xi) = N^{-2} \vp''(N^{-1}\xi).\]
	Note that although $m_1$ depends on the parameter $N$, the constant $C$ in (\ref{m1_db}) needs to be independent of $N$.
	\begin{itemize}
		\item Case 1: $0 < |\xi| \leq N$. In this case, (\ref{m1_db}) holds trivially since $m_1'(\xi) = m_1''(\xi) = 0$.
		
		\item Case 2: $N \leq |\xi| \leq 2N$. 
		
		Firstly, since $\vp$ is a fixed smooth function on $\m{R}$, there exists an absolute constant $C_0$ such that $|\vp'(\xi)| + |\vp''(\xi)| \leq C_0$ for any $1\leq |\xi| \leq 2$. As a result, 
		\[
			|m_1'(\xi)| = N^{-1} |\vp'(N^{-1}\xi)| \leq C_0 N^{-1}, \quad |m_1''(\xi)| = N^{-2} |\vp''(N^{-1}\xi)| \leq C_0 N^{-2}.
		\]
		Meanwhile, since $\vp$ is decreasing and even, 
		\[
			\frac{m_1(\xi)}{|\xi|} = \frac{\vp(N^{-1}\xi)}{|\xi|} \geq \frac{\vp(2)}{2N} = \frac{2^{s-2}}{N}, 
			\qquad 
			\frac{m_1(\xi)}{|\xi|^2} \geq \frac{\vp(2)}{(2N)^2} = \frac{2^{s-3}}{N^2}.
		\]
		Therefore, (\ref{m1_db}) is satisfied by requiring $C \geq 2^{3-s} C_0$.
		
		\item Case 3: $|\xi| \geq 2N$. 
		
		In this case, $|N^{-1}\xi| \geq 2$, so $m_1(\xi) = \vp(N^{-1}\xi) = N^{1-s} |\xi|^{s-1}$, which implies that 
		\[
			|m_1'(\xi)| = (1-s) N^{1-s} |\xi|^{s-2}, \qquad  |m_1''(\xi)| = (1-s)(2-s) N^{1-s} |\xi|^{s-3}.
		\]
		Hence, (\ref{m1_db}) holds by requiring $C \geq (1-s)(2-s)$.
	\end{itemize}
Combining the above three cases, (\ref{m1_db}) holds by choosing $C = 2^{3-s}C_0 + (1-s)(2-s)$.
\end{proof}

\begin{lemma}\label{Lemma, m1p2}
	Let $s\in [\frac34, 1)$. If $\frac{1}{C}|\xi| \leq |\eta| \leq C |\xi|$ for some $C\geq 1$, then 
	\be\label{m1_ineq}
		(2C)^{s-1} m_1(\eta) \leq m_1(\xi) \leq (2C)^{1-s} m_1(\eta).
	\ee
\end{lemma}
\begin{proof}
	Let $s\in [\frac34, 1)$ and $C \geq 1$. We first prove that if $|\eta| \leq C |\xi|$, then
	\be\label{vp_compa}
		\vp(\xi) \leq (2C)^{1-s} \vp(\eta).
	\ee
	Without loss of generality, we can assume that $|\eta|\geq |\xi|$, otherwise (\ref{vp_compa}) holds trivially since $\vp$ is a decreasing function.
	\begin{itemize}
		\item Case 1: $|\xi|\leq \frac{1}{C}$. In this case, (\ref{vp_compa}) holds automatically since $\vp(\xi) = \vp(\eta) = 1$.
		
		\item Case 2: $|\xi| \geq 2$. In this case, $\vp(\eta) \geq \vp(C\xi) = C^{s-1}|\xi|^{s-1}$, so 
		\[
			\vp(\xi) = |\xi|^{s-1} \leq C^{1-s} \vp(\eta),
		\]
		which justifies (\ref{vp_compa}).
		
		\item Case 3: $\frac{1}{C} \leq |\xi| \leq 2$. In this case, $|\eta| \leq C|\xi| \leq 2C$, so 
		$\vp(\eta) \geq \vp(2C) = (2C)^{s-1}$. Therefore, (\ref{vp_compa}) holds since $\vp(\xi)\leq 1$.
	\end{itemize}
Noting that $m_{1}(x) = \vp(N^{-1}x)$ for any $x\in\m{R}$, so (\ref{m1_ineq}) follows from (\ref{vp_compa}).
\end{proof}

Since the range for $s$ in this paper is restricted to $[\frac34, 1)$, for convenience of notations, we write $f \ls g$ if $|f| \leq C |g|$, where $C$ is some positive constant only depending on $s$. The notation $f \gs g$ can be interpreted in a similar way, and $f \sim g$ means that both $f\ls g$ and $f\gs g$ hold. With this convention, Lemma \ref{Lemma, m1p2} implies that 
\be\label{m1_compa}
	m_1(\xi) \sim m_1(\eta) \quad \text{whenever} \quad \xi \sim \eta.
\ee

\subsection{Key estimate on the multiplier $\sigma_3$}
With the choice of $m_2$ in (\ref{m2_fix}), we will prove that $\sigma_3$ is bounded pointwise in Lemma \ref{Lemma, est on sigma_3} below. 

%Actually, according to (\ref{sigma_3_cpt}), for any $(\xi_1,\xi_2,\xi_3)\in\Gamma_{3}$ with $\max_{1\leq i\leq 3} |\xi_i| \leq N/2$, it holds that $\sigma_3(\xi_1,\xi_2,\xi_3) = 1$. So it suffices to consider the case when $\max_{1\leq i\leq 3} |\xi_i| \geq N/2$.

\begin{lemma}\label{Lemma, est on sigma_3}
Let $\frac{3}{4} \leq s < 1$ and $N \geq 10$. Define $m_2$ as in (\ref{m2_fix}) and $\sigma_3$ as in (\ref{sigma_3_cpt}). Let
\[(\xi_1, \xi_2, \xi_3)\in\Gamma_3, \quad N_i = |\xi_i|, \quad\forall\, i=1,2,3.\]
Then $\sigma_3(\xi_1, \xi_2, \xi_3)$ is pointwisely bounded as stated below.
\begin{enumerate}
	\item[(1)] If $ \max\limits_{1\leq i\leq 3} |\xi_i| \leq N/2 $, then $\sigma_3(\xi_1,\xi_2,\xi_3) = 1$;
	\item[(2)] If $ \max\limits_{1\leq i\leq 3} |\xi_i| \geq N/2 $ and 
	\begin{enumerate}
		\item $|\xi_2 - \xi_3| \leq 10 |\xi_1|$, then $|\sigma_3(\xi_1,\xi_2,\xi_3)| \leq C N_1^{2s-2} N^{2-2s}$;
		
		\item $|\xi_1| \leq \frac{1}{10} |\xi_2 - \xi_3|$, then $|\sigma_3(\xi_1,\xi_2,\xi_3)| \leq  C (\max\{N_2,N_3\})^{2s-2} N^{2-2s}$;
	\end{enumerate}
\end{enumerate}
where $C$ is some constant only depending on $s$. Thus, 
\be\label{est on sigma_3}
	|\sigma_3(\xi_1, \xi_2, \xi_3)| \leq C N_{max}^{2s-2} N^{2-2s} \leq C, 
\ee
where $N_{max} = N + \max\{N_1,N_2,N_3\}$.
\end{lemma}

\begin{proof}
Firstly, it is readily seen that the estimate (\ref{est on sigma_3}) follows immediately from the two statements (1) and (2). Secondly, statement (1) can be justified directly from the definitions of $\sigma_3$, $m_1$ and $m_2$ in (\ref{sigma_3_cpt}), (\ref{m1}) and (\ref{m2_fix}) respectively.
So it remains to prove statement (2) in which $\max\limits_{1\leq i\leq 3} |\xi_i| \geq N/2 $.

Denote the numerator and the denominator of $\sigma_3$ to be $\eta_3$ and $\alpha_3$ respectively. Since $m_2(\xi) = m_1(2\xi)$, then 
\be\label{eta3}
	\eta_3 = \xi_1^3 m_1^2(\xi_1) + 4\xi_2^3 m_1^2(2\xi_2) + 4\xi_3^3 m_1^2(2\xi_3).
\ee
In addition, we recall that $\alpha_3 = \xi_1^3 + 4\xi_2^3 + 4\xi_3^3 = -3\xi_1 (\xi_2 - \xi_3)^2$.
The most challenging case is when $|\xi_2-\xi_3|$ is small, so we denote $\delta = \xi_2 - \xi_3$ to track this difference so that $\a_3 = -3\xi_1\delta^2$.
Since $\xi_1 = - \xi_2 - \xi_3$, we find 
\be\label{xi_delta}
	\xi_2 = -\frac12 (\xi_1 - \delta), \quad \xi_3 = -\frac12 (\xi_1 + \delta).
\ee
Next, we divide the proof into three cases.
\begin{itemize}
\item \textbf{Case 1:} $|\xi_2 - \xi_3| \leq \frac{1}{10} |\xi_1|$. \\
In this case, $|\delta| \leq \frac{1}{10} |\xi_1|$, so $|\xi_2| \sim |\xi_3| \sim |\xi_1| \sim N_1$. Substituting \eqref{xi_delta} into $\eta_3$ yields
\[
\eta_3 = -\frac{1}{2} \left[ f(\xi_1 + \delta) - 2f(\xi_1) + f(\xi_1 - \delta) \right],
\]
where $f(\xi) := \xi^3 m_1^2(\xi)$.
By Taylor expansion, 
\[ 
	|\eta_3| \leq \frac{1}{2} |f''(\xi_1 + \theta \delta)| \delta^2, \quad \text{for some}\,\, \theta \in [-1,1].
 \] 
Since $|\xi_1 + \theta \delta| \sim |\xi_1| \gs N$, it then follows from Lemma \ref{Lemma, m1p1} and Lemma \ref{Lemma, m1p2} that
\[ m_1(\xi_1+\th\delta) \ls m_1(\xi_1), \quad |m_{1}'(\xi_1 + \th\delta)| \ls \frac{m_1(\xi_1)}{|\xi_1|}, \quad |m_{1}''(\xi_1 + \th\delta)| \ls \frac{m_1(\xi_1)}{|\xi_1|^2}.\]
Therefore, $|f''(\xi_1 + \theta \delta)| \ls |\xi_1| m_{1}^{2}(\xi_1)$ and
\[
	|\sigma_3| = \Big| \frac{\eta_3}{\a_3} \Big| \ls \frac{|\xi_1| m_{1}^{2}(\xi_1) \delta^2}{3 |\xi_1| \delta^2} = \frac13 m_{1}^{2}(\xi_1) \sim m_1^2(N_1) \ls N_1^{2s-2} N^{2-2s}.
\]

\item \textbf{Case 2:} $ \frac{1}{10} |\xi_1| \leq |\xi_2 - \xi_3| \leq 10 |\xi_1|$. \\
In this case, $|\delta| \sim |\xi_1|$, so it follows from (\ref{xi_delta}) that $|\xi_2| \ls |\xi_1|$ and $|\xi_3| \ls |\xi_1|$. As a result, it follows from (\ref{eta3}) and Lemma \ref{Lemma, m1p2} that
\[
|\eta_3| \ls N_1^3 m_{1}^{2}(N_1), \quad |\a_3| \sim |\xi_1| \delta^2 \sim N_1^3.
\]
Thus, $|\sigma_3| \ls m_{1}^{2}(N_1) \ls N_1^{2s-2} N^{2-2s}$.

\item \textbf{Case 3:} $ |\xi_2 - \xi_3| \geq 10 |\xi_1| $. \\
In this case, $|\xi_1| \leq |\xi_2 - \xi_3|/10 = |\delta|/10$, then it follows from (\ref{xi_delta}) that $|\xi_2|\sim |\xi_3| \sim |\delta| \gg |\xi_1|$, and $\xi_2$ and $\xi_3$ have opposite signs. 
Meanwhile, since $\xi_2$ and $\xi_3$ are symmetric, we can assume that $N_2 \geq N_3$, so it reduces to prove that $|\sigma_3| \ls N_2^{2s-2} N^{2-2s}$. We split $\eta_3$ into three pieces:
\[
	\eta_3 = \underbrace{\xi_1^3 m_1^2(\xi_1)}_{I_1} + \underbrace{(4\xi_2^3 + 4\xi_3^3) m_1^2(2\xi_2)}_{I_2} + \underbrace{4\xi_3^3 \left[ m_1^2(2\xi_3) - m_1^2(2\xi_2)\right]}_{I_3}.
\]
Estimating each term:
\[\begin{array}{l}
	|I_1| = N_1^3 m_{1}^{2}(N_1), \vspace{0.1in} \\
	|I_2| = \big| 4(\xi_2+\xi_3)(\xi_2^2 - \xi_2\xi_3 + \xi_3^2)m_1^2(2\xi_2) \big| \sim N_1 N_2^2 m_{1}^{2}(2N_2),  \vspace{0.1in} \\
	|I_3| \sim N_3^3 \big[ m_1(2\xi_3) + m_1(2\xi_2) \big] \big| m_1(2\xi_3) - m_1(2\xi_2) \big|.
\end{array}\]
Since $m_1$ is an even function and $\xi_2$ and $\xi_3$ have opposite signs, combining this with Lemma \ref{Lemma, m1p1} and Lemma \ref{Lemma, m1p2} yields
\[\begin{split}
	\big| m_1(2\xi_3) - m_1(2\xi_2) \big| &= \big| m_1(2\xi_3) - m_1(-2\xi_2) \big| \\
	&\sim |m_1'(2N_2)| |\xi_3+\xi_2| \sim N_2^{-1} m_1(2N_2) N_1.
\end{split}\]
Thus, $|I_3| \sim N_3^3 N_2^{-1} m_1^2(2N_2) N_1$ and 
\[ |\eta_3| \ls N_1^3 m_{1}^{2}(N_1) + N_1 N_2^2 m_{1}^{2}(2N_2) + N_3^3 N_2^{-1} m_1^2(2N_2) N_1 \ls N_1 N_2^2 m_{1}^{2}(2N_2). \]
On the other hand, 
\[|\a_3| = |3\xi_1(\xi_2-\xi_3)|^2 \sim N_1 N_2^2. \]
Consequently, 
\[
	|\sigma_3| \sim \frac{|\eta_3|}{|\a_3|} \ls m_1^2(2N_2) \ls N^{2-2s} N_2^{2s-2}. 
\]
\end{itemize}

Combining the above Case 1--Case 3, Lemma \ref{Lemma, est on sigma_3} is justified.
\end{proof}

\subsection{Estimate on the multipliers $M_{41}$ and $M_{42}$}
Thanks to Lemma \ref{Lemma, est on sigma_3}, we are ready to address the estimates for the multipliers $M_{41}$ and $M_{42}$.
\begin{lemma}\label{Lemma, est on M4}
Let $\frac{3}{4} \leq s < 1$ and $N \geq 10$. Let $M_{41}$ and $M_{42}$ be as defined in  (\ref{mpl_4}). Then
\[
	|M_{41}(\xi_1,\xi_2,\xi_3,\xi_4)| + |M_{42}(\xi_1,\xi_2,\xi_3,\xi_4)| \leq C N_{max}^{2s-1} N^{2-2s}, \quad \forall\, (\xi_1, \xi_2, \xi_3, \xi_4)\in\Gamma_4,
\]
where $N_{max} := N + \max\limits_{1 \leq i \leq 4} |\xi_i|$, and $C$ is a constant only depending on $s$.
\end{lemma}

\begin{proof}
According to (\ref{mpl_4}) and the symmetry of $\sigma_{3}$ in its last two variables, we have 
\[ 
	M_{41} = \frac{1}{2} (\xi_1 + \xi_4) \sigma_3(\xi_1 + \xi_4, \xi_2, \xi_3), \quad M_{42} = 2  (\xi_2 + \xi_3) \sigma_3(\xi_1, \xi_2 + \xi_3, \xi_4).
\]
We first handle $M_{41}$. 
If $\max\{ |\xi_1 + \xi_4|, |\xi_2|, |\xi_3| \} \leq \frac{N}{2}$, then it follows from Part (1) in Lemma \ref{Lemma, est on sigma_3} that $\sigma_3(\xi_1 + \xi_4, \xi_2, \xi_3) = 1$. As a result,
\[ 
	|M_{41}(\xi_1,\xi_2,\xi_3,\xi_4)| = \frac12 |\xi_1 + \xi_4| \leq \frac{N}{4} \leq \frac14 N_{max}^{2s-1} N^{2-2s}.
\]	
Next, we assume $\max\{ |\xi_1 + \xi_4|, |\xi_2|, |\xi_3| \} \geq \frac{N}{2}$ and divide the rest proof into two cases.
\begin{itemize}
	\item \textbf{Case 1:} $|\xi_2 - \xi_3| \leq 10|\xi_1 + \xi_4|$. \\
	In this case, it follows from Lemma \ref{Lemma, est on sigma_3} that 
	\[ |M_{41}| \ls |\xi_1 + \xi_4| |\xi_1 + \xi_4|^{2s-2} N^{2-2s} =  |\xi_1 + \xi_4|^{2s-1} N^{2-2s}.\]
	Since $s\geq 3/4$, then $2s-1>0$, so $|M_{41}| \ls N_{max}^{2s-1} N^{2-2s}$.
	
	\item \textbf{Case 2:} $ |\xi_1 + \xi_4| \leq \frac{1}{10}|\xi_2 - \xi_3|$. \\
	In this case, it again follows from Lemma \ref{Lemma, est on sigma_3} and $2s-2<0$ that 
	\[|M_{41}| \ls |\xi_1 + \xi_4| |\xi_2|^{2s-2} N^{2-2s}. \]
	Noting 
	\be\label{eta2} 
		\xi_2 = \frac12 \big[ (\xi_2+\xi_3) + (\xi_2-\xi_3) \big] = \frac12 \big[ -(\xi_1+\xi_4) + (\xi_2-\xi_3) \big],
	\ee
	so we take advantage of the assumption $ |\xi_1 + \xi_4| \leq \frac{1}{10} |\xi_2 - \xi_3|$ and (\ref{eta2}) to obtain
	\[|\xi_2| \geq |\xi_2 - \xi_3|/3 \geq 3|\xi_1 + \xi_4|.\] 
	Hence, 
	\[|M_{41}| \ls |\xi_2|^{2s-1} N^{2-2s} \leq N_{max}^{2s-1} N^{2-2s}. \]
\end{itemize}

We next analyze $M_{42}$. Similarly to the discussion in the $M_{41}$ case, it suffices to consider the case when $\max\{|\xi_1|, |\xi_2+\xi_3|, |\xi_4|\} \geq \frac{N}{2}$. We then again divide the rest argument into two parts.
\begin{itemize}
	\item \textbf{Case 1:} $ |\xi_2 + \xi_3 - \xi_4| \leq 10|\xi_1|$. \\
	In this case, it follows from Lemma \ref{Lemma, est on sigma_3} that 
	\[ |M_{42}| \ls |\xi_2+\xi_3| |\xi_1|^{2s-2} N^{2-2s}. \]
	Since $|\xi_2+\xi_3-\xi_4| \ls |\xi_1|$, then 
	\[\begin{split}
		|\xi_2+\xi_3| &= \big| \frac12(\xi_2+\xi_3+\xi_4) + \frac12 (\xi_2+\xi_3-\xi_4) \big| \\
		& = \big| -\frac12 \xi_1 + \frac12 (\xi_2+\xi_3-\xi_4) \big| \ls |\xi_1|,
	\end{split}\]
	which implies that $|M_{42}| \ls |\xi_1|^{2s-1} N^{2-2s} \ls N_{max}^{2s-1} N^{2-2s}$.
	
	\item \textbf{Case 2:} $|\xi_1| \leq \frac{1}{10} |\xi_2 + \xi_3 - \xi_4|$. \\
	In this case, it follows from Lemma \ref{Lemma, est on sigma_3} that 
	\[ |M_{42}| \ls |\xi_2+\xi_3|  |\xi_2+\xi_3|^{2s-2} N^{2-2s} =  |\xi_2+\xi_3|^{2s-1} N^{2-2s} \ls N_{max}^{2s-1} N^{2-2s}. \]
\end{itemize}
Combining all the above situations together justifies Lemma \ref{Lemma, est on M4}.
\end{proof}

\section{Growth estimate of modified energies}
\label{Sec, est_me}
\subsection{Lifespan of local solutions}
For $s\in \big[ \frac34, 1 \big)$, it has been shown in \cite{YZ22a} that the Majda-Biello system (\ref{MB-4}) with the dispersion coefficient $\a=4$ is locally well-posed for any $s\geq \frac34$. The local solution lives in the Fourier restriction space whose definition is given below. 
\begin{definition}\label{Def, FR space on R}
	For any $\alpha, s, b \in \mathbb{R}$ with $\alpha \neq 0$, the Fourier restriction space $X^{\alpha}_{s, b}(\mathbb{R}^2)$ is defined as the completion of the Schwartz space $\mathscr{S}(\mathbb{R}^{2})$ with the norm
	\be\label{FR norm on R}
	\|w\|_{ X^{\a}_{s,b}( \m{R}^2 ) } := \|\la \xi \ra ^{s} \la \tau - \a \xi^3 \ra^{b} \wh{w}(\xi,\tau)\|_{ L^{2}_{\xi\tau} (\m{R}^2) },
	\ee
	where $\la \cdot \ra = 1 + |\cdot|$, and $\wh{w}$ refers to the space-time Fourier transform of $w$. 
	Moreover, for any $T>0$,  $X^{\a}_{s,b} (\m{R} \times [0,T])$ denotes the
	restriction of $X^{\a}_{s,b}(\m{R}^2)$ on the domain $\m{R} \times [0,T]$ which is a Banach space when equipped with the usual quotient norm.
\end{definition}
According to the above notations and \cite{YZ22a}, the local solution $(u,v)$ of (\ref{MB-4}) lives in the space $X_{s,b}^{1}(\m{R} \times [0,T]) \times X_{s,b}^{4}(\m{R} \times [0,T])$ for any initial data $(u_0, v_0)\in \mathcal{H}^{s}$. Since the modified energies in the previous section involve $I$-operators, we next provide another version of the local well-posedness theory in which the $I$-operators play a role. This result will lay the foundation for estimating the growth rate of the modified energy $E^{(2)}(u,v)$. 

\begin{proposition}\label{Prop, ls}
Let $\frac34\leq s < 1$, $\frac12 < b \leq 1$ and $(u_0, v_0) \in \mathcal{H}^{s}(\m{R})$. Then there exists a constant $\eps_{*}$, which only depends on $s$ and $b$, such that if 
\be\label{initial_small}
	F_{0} := \|(I_1 u_0, I_2 v_0)\|_{\mathcal{H}^{1}(\m{R})} \leq \eps_{*},
\ee 
then the lifespan of the local solution $(u,v)$ of (\ref{MB-4}) is at least $1$ and 
\be\label{ub_ls}
	\|I_1 u\|_{X^{1}_{1,b} (\m{R} \times [0, 1])} + \|I_2 v\|_{X^{4}_{1,b} (\m{R} \times [0, 1])} \leq C F_{0},
\ee
where $C$ is a positive constant which only depends on $s$ and $b$.
\end{proposition}

\begin{proof}
Let $\psi \in C^{\infty}(\m{R})$ be a cut-off function such that $\psi = 1$ on $[-1,1]$ and $\text{supp } \psi \subseteq [-2,2]$. For $\a = 1$ or $4$, we denote $S_{\a}$ to be the semigroup operator that maps any initial data $w_0\in H^{s}(\m{R})$ to be the solution $w\in C(\m{R}; H^s (\m{R}))$ of the following KdV equation (\ref{lin_eq}), that is $w(x,t) = [S_{\a}(t) w_0](x)$ or simply $w(t) = S_{\a}(t) w_0$.
\be\label{lin_eq}
\left\{\begin{array}{ll}
	w_{t} + \alpha w_{xxx} + w_{x} = 0, & x\in\m{R},\, t\in\m{R},\\
	w(x,0)=w_{0}(x).
\end{array}\right.
\ee	
Then the solution of the Cauchy problem (\ref{MB-4}) for $t\in[0,1]$ can be regarded as the restriction of the solution of the following integral equations to the time interval $[0,1]$.
\be\label{Duh_soln}\left\{ 
\begin{aligned} 
	u(t) &= \psi(t) \Big( S_{1}(t) u_{0} - \int_{0}^{t} S_{1}(t - \tau) (v v_x)(\tau) \,d\tau \Big),\\ 
	v(t) &= \psi(t) \Big(S_{4}(t) v_{0} - \int_{0}^{t} S_{4}(t - \tau) (uv)_{x}(\tau) \,d\tau\Big).
\end{aligned}
\right.\ee
By applying the $I$-operators $I_1$ and $I_2$ to the above two equations respectively, it reduces to solve 
\be\label{integ form}\left\{ 
\begin{aligned} 
	I_1 u(t) &= \psi(t) \Big( S_{1}(t) I_1 u_{0} - \int_{0}^{t} S_{1}(t - \tau) I_1(v v_x)(\tau) \,d\tau \Big),\\ 
	I_2 v(t) &= \psi(t) \Big(S_{4}(t) I_2 v_{0} - \int_{0}^{t} S_{4}(t - \tau) I_2[(uv)_{x}](\tau) \,d\tau\Big).
\end{aligned}
\right.\ee
Define 
\[f(t) = I_1 u(t), \quad f_0 = I_1 u_0, \quad g(t) = I_2 v(t), \quad g_0 = I_2 v_0.\]
Since $m_1$ and $m_2$ are bounded and positive functions, $I_1$ and $I_2$ have inverses which are denoted as $I_1^{-1}$ and $I_2^{-1}$. Hence, (\ref{integ form}) is equivalent to 
\be\label{integ form_I}\left\{ 
\begin{aligned} 
	f(t) &= \psi(t) \Big( S_{1}(t) f_{0} - \int_{0}^{t} S_{1}(t - \tau) I_1(v v_x)(\tau) \,d\tau \Big),\\ 
	g(t) &= \psi(t) \Big(S_{4}(t) g_{0} - \int_{0}^{t} S_{4}(t - \tau) I_2[(uv)_{x}](\tau) \,d\tau\Big),
\end{aligned}
\right.\ee
where $u = I_{1}^{-1} f$ and $v = I_2^{-1} g$.

This suggests to consider the map $\Phi(f,g)\triangleq \big(\Phi_{1}(f,g), \Phi_{2}(f,g)\big)$, where 
\be\label{contra map}
\left\{\begin{aligned}
	\Phi_{1}(f,g) &:= \psi(t) \Big( S_{1}(t) f_{0} - \int_{0}^{t} S_{1}(t - \tau) I_1(v v_x)(\tau) \,d\tau \Big),\\
	\Phi_{2}(f,g) &:= \psi(t) \Big(S_{4}(t) g_{0} - \int_{0}^{t} S_{4}(t - \tau) I_2[(uv)_{x}](\tau) \,d\tau\Big),
\end{aligned}\right. \ee
and then show $\Phi $ is a contraction mapping in a ball in an appropriate Banach space, which will imply that the fixed point of $\Phi$ is the desired solution to the Cauchy problem (\ref{MB-4}) for $0\leq t\leq 1$.
Let $Y_{b} := X^1_{1,b}\times X^4_{1,b} $ be equipped with the norm  
$\|(f, g)\|_{Y_{b}} := \| f \|_{X^{1}_{1,b}} + \| g \|_{X^{4}_{1,b}}$.
Define 
\be\label{soln space}
\mcal{B}_{C}(f_0, g_0) = \{(f,g)\in Y_{b}: \|(f,g)\|_{Y_{b}} \leq C F_{0} \},
\ee
where $F_0 =  \|(I_1 u_0, I_2 v_0)\|_{\mathcal{H}^{1}(\m{R})} := \|f_0\|_{H^1} + \|g_0\|_{H^1}$.
In the following, we will choose suitable $\eps_{*}$ in (\ref{initial_small}) and a constant $C$  in (\ref{soln space}) such that $\Phi$ is a contraction mapping on $\mcal{B}_{C}(f_0, g_0)$. 

We will first show that $\Phi$ maps  the closed ball $\mcal{B}_{C} (f_0, g_0)$ into itself. For any  $(f,g)\in \mcal{B}_{C} (f_0, g_0)$,  by standard estimates in the Fourier restriction spaces (see e.g. Lemma 3.1 and Lemma 3.3 in \cite{KPV93Duke}), we obtain
\be\label{est on Phi}\left\{\begin{aligned}
	\|\Phi_{1}(f,g)\|_{X^{1}_{1,b}} & \leq C_1 \|f_0\|_{H^{1}} + C_1 \| I_1(v v_x)\|_{X^{1}_{1,b-1}},\\
	\|\Phi_{2}(f,g)\|_{X^{4}_{1,b}} & \leq C_1 \|g_0\|_{H^{1}} + C_1 \|I_{2}[(uv)_{x}]\|_{X^{4}_{1,b-1}}.
\end{aligned}\right.\ee
Now by taking advantage of the bilinear estimates which involve the $I$-operators in Lemma \ref{Lemma, bl_est_I} (with $\sigma_1 = 0$), we find
\be\label{nonlin est}\left\{\begin{aligned}
& \| I_1(v v_x) \|_{X^{1}_{1, b-1}} \leq C_{2} \|I_2 v\|_{X^{4}_{1,b}}^{2} = C_{2} \| g \|_{X^{4}_{1,b}}^{2} \\
& \| I_{2}[(uv)_{x}] \|_{X^{4}_{1, b-1}} \leq C_{2} \|I_1 u\|_{X^{1}_{1,b}} \|I_2 v\|_{X^{1}_{1,b}} = C_{2} \| f \|_{X^{1}_{1,b}} \| g \|_{X^{4}_{1,b}}.
\end{aligned}\right.\ee
Plugging (\ref{nonlin est}) into (\ref{est on Phi}) yields
\be\label{apriori bdd for Phi}\begin{split}
	\|\Phi(f,g)\|_{Y_b} &\leq C_{3} (\|f_0\|_{H^1} + \|g_0\|_{H^1}) + C_3 \big(\| f \|_{X^{1}_{1,b}} + \| g \|_{X^{4}_{1,b}} \big)^2, \\
	& = C_{3} F_0 + C_3 \| (f,g) \|_{Y_b}^2.
\end{split}\ee
Define $C^{*} = 8 C_{3}$. Then it follows from (\ref{soln space}) and (\ref{apriori bdd for Phi}) that for any  $(f,g)\in \mcal{B}_{C^*} (f_0, g_0)$, we have $\|(f,g)\|_{Y_{b}} \leq C^{*} F_{0}$ and 
\[
	\|\Phi(f, g)\|_{Y_b} \leq C_{3} F_{0} + C_{3} (C^{*})^{2} F_{0}^{2} = C_{3} F_{0} + 64 C_{3}^{3} F_0^2.
\]
Choosing 
\be\label{choice of eps}
	\eps_{*} = \frac{1}{64 C_3^2},
\ee
then whenever $F_0$ satisfies (\ref{initial_small}), we have 
\[
	\|\Phi(f,g)\|_{Y_b} \leq 2C_{3} F_{0}= \frac{C^{*}F_0}{4},
\]
which justifies $\Phi(f,g)\in \mcal{B}_{C^*} (f_0, g_0)$. 

Next for any $(f_j, g_j)\in \mcal{B}_{C^*} (f_0, g_0)$, $j=1,2$, the same argument yields 
\[ \| \Phi (f_1, g_1 ) - \Phi (f_2 , g_2)\|_{Y_b} \leq \frac12 \| (f_1, g_1) - (f_2, g_2)\|_{Y_b}.\]
We  have thus shown  that $\Phi$ is a contraction on $\mcal{B}_{C^*} (f_0, g_0)$ and its fixed point $(f,g)$ solves the equation (\ref{integ form_I}). As a result, $(u,v) := (I_1^{-1}f, I_2^{-1}g)$ solves the equation (\ref{integ form}) as well as (\ref{Duh_soln}). Finally, since the function $\psi$ in (\ref{Duh_soln}) is identically equal to $1$ on $[0,1]$, the restriction of $(u, v)$ to $\m{R}\times [0,1]$ solves the original system (\ref{MB-4}) for $t\in[0,1]$. Finally, it follows from the definition of (\ref{soln space}) that
\[
	\|I_1 u\|_{X^{1}_{1,b} (\m{R} \times [0, 1])} + \|I_2 v\|_{X^{4}_{1,b} (\m{R} \times [0, 1])} 
	\leq \| f \|_{X^{1}_{1,b} (\m{R}^2 )} + \| g \|_{X^{4}_{1,b} (\m{R}^2)} \leq C^{*} F_0,
\]
the proof of Proposition \ref{Prop, ls} is complete.
\end{proof}

Proposition \ref{Prop, ls} requires the initial data to be small as shown in (\ref{initial_small}), this restriction can be removed by a standard scaling argument and it turns out that the lifespan of the local solution is inverse proportional to a superlinear power of the size of the initial data, see the next result for more precise description.

\begin{proposition}\label{Prop, ls1}
	Let $\frac34\leq s < 1$, $\frac12 < b \leq 1$ and $(u_0, v_0) \in \mathcal{H}^{s}(\m{R})$. Denote \[F_{0} = \| (I_1 u_0, I_2 v_0) \|_{\mathcal{H}^{1}(\m{R})}.\]
	Then the lifespan of the local solution $(u,v)$ has a lower bound 
	$T_{*} := \delta F_{0}^{-6/(2s+1)}$, 
	where $\delta$ is a positive constant only depending on $s$ and $b$. Meanwhile, 
	\be\label{ub_ls1}
	\|I_1 u\|_{X^{1}_{1,b} (\m{R} \times [0,T_{*}])} + \|I_2 v\|_{X^{4}_{1,b} (\m{R} \times [0,T_{*}])} \leq 
	C F_{0}^{\frac{6b+2}{2s+1}},
	\ee
	where $C$ is a positive constant only depending on $s$ and $b$.
\end{proposition}
\begin{proof}
	The local well-posedness for (\ref{MB-4}) has been treated in Proposition \ref{Prop, ls} for small initial data, so we will first scale $(u,v)$ so that its initial data $(u_0,v_0)$ becomes small. More precisely, for any $\lam\geq 1$, we define $(u^{\lam}, v^{\lam})$ as 
	\be\label{scaled_soln}
		(u^{\lam}, v^{\lam})(x,t) := \lam^{-2} (u,v)(\lam^{-1}x, \lam^{-3} t).
	\ee
	Then $(u,v)$ solves (\ref{MB-4}) with initial data $(u_0, v_0)$ if and only if $(u^{\lam}, v^{\lam})$ solves the same equation with a scaled initial data, that is,
	\be\label{MB-4-lam}\left\{
	\begin{array}{rcl}
		u^{\lam}_{t} + u^{\lam}_{xxx} & = & - v^{\lam} v^{\lam}_x,\\
		v^{\lam}_{t} + 4 v^{\lam}_{xxx} & = & - (u^{\lam}v^{\lam})_{x},\\
		(u^{\lam}, v^{\lam})|_{t=0} & = & (u^{\lam}_0,v^{\lam}_0),
	\end{array} \right. \ee
	where $(u^{\lam}_0, v^{\lam}_0)(x) := \lam^{-2} (u_0, v_0)(\lam^{-1}x)$. Next, we compare $\| (I_1 u^{\lam}_{0}, I_2 v^{\lam}_{0}) \|_{\mathcal{H}^{1}(\m{R})} $ with $\| (I_1 u_0,  I_2 v_0) \|_{\mathcal{H}^{1}(\m{R})}$. Firstly, $\wh{I_1 u_0^{\lam}}(\xi) = \lam^{-1} m_1(\xi) \wh{u_0}(\lam\xi)$, so 
	\[
		\| I_1 u^{\lam}_{0} \|_{H^{1}(\m{R})}^2 =  \lam^{-2} \int_{\m{R}} (1+|\xi|)^2 m_1^2(\xi) |\wh{u_0}(\lam\xi)|^2 \,d\xi.
	\]
	By the change of variable $\eta := \lam\xi$, we obtain 
	\[
		\| I_1 u^{\lam}_{0} \|_{H^{1}(\m{R})}^2 =  \lam^{-3} \int_{\m{R}} (1 + \lam^{-1}|\eta|)^2 m_1^2(\lam^{-1}\eta) |\wh{u_0}(\eta)|^2 \,d\eta.
	\]
	Since $\lam\geq 1$ and $s\in[\frac34, 1)$, then $(1 + \lam^{-1}|\eta|)^2 \leq (1 + |\eta|)^2$ and $m_1(\lam^{-1}\eta) \leq (2\lam)^{1-s} m_1(\eta)$ due to Lemma \ref{Lemma, m1p2}. Consequently, 
	\[
		\| I_1 u^{\lam}_{0} \|_{H^{1}(\m{R})}^2 \leq \lam^{-3}  (2\lam)^{2-2s} \int_{\m{R}} (1 + |\eta|)^2 m_1^2(\eta) |\wh{u_0}(\eta)|^2 \,d\eta = 2^{2-2s} \lam^{-1-2s}\| I_1 u_{0} \|_{H^{1}(\m{R})}^2.
	\] 
	Similarly, we can establish a similar relation between $\| I_1 v^{\lam}_{0} \|_{H^{1}(\m{R})}$ and $\| I_1 v_{0} \|_{H^{1}(\m{R})}$. Thus, 
	\be \label{scaled_norm}
	\| (I_1 u^{\lam}_{0}, I_2 v^{\lam}_{0}) \|_{\mathcal{H}^{1}(\m{R})} \leq 2^{1-s} \lam^{-\frac12-s} \| (I_1 u_0,  I_2 v_0) \|_{\mathcal{H}^{1}(\m{R})} \leq 2\lam^{-\frac12-s} F_0, 
	\ee
	
	Thanks to (\ref{scaled_norm}), as long as the parameter $\lam$ is chosen as 
	\be\label{choice_lam}
		\lam = 2 \eps_{*}^{-\frac{2}{2s+1}} F_{0}^{\frac{2}{2s+1}},
	\ee
	where $\eps_*$ is the constant in Proposition \ref{Prop, ls},  
	then the restriction (\ref{initial_small}) holds for $(I_1 u_0^{\lam}, I_2 v_0^{\lam})$. Hence, it follows from Proposition \ref{Prop, ls} that (\ref{MB-4-lam}) has a local solution $(u^{\lam},v^{\lam})$ on the time interval $[0,1]$ such that 
	\be\label{scaled_fr_est}
	\|I_1 u^{\lam}\|_{X^{1}_{1,b} (\m{R} \times [0, 1])} + \|I_2 v^{\lam}\|_{X^{4}_{1,b} (\m{R} \times [0, 1])} \leq 
	C \| (I_1 u^{\lam}_{0}, I_2 v^{\lam}_{0}) \|_{\mathcal{H}^{1}(\m{R})}.
	\ee
	By converting $(u^{\lam}, v^{\lam})$ back to $(u,v)$, we find (\ref{MB-4}) admits a local solution $(u,v)$ on the time interval $[0, \lam^{-3}]$ satisfying $I_1 u\in X^{1}_{1,b} (\m{R} \times [0,\lam^{-3}])$ and $I_2 v\in X^{4}_{1,b} (\m{R} \times [0,\lam^{-3}])$. In addition,
	we will take advantage of (\ref{scaled_fr_est}) and (\ref{scaled_norm}) to find a relation between $\|I_1 u\|_{X^{1}_{1,b}} + \|I_2 v\|_{X^{4}_{1,b}} $ and $F_0$. Firstly, based on (\ref{scaled_soln}) and the space-time Fourier transform, we know $\wh{I_1 u^{\lam}}(\xi,\tau) = \lam^2 m_1(\xi) \wh{u}(\lam\xi, \lam^3\tau)$. So 
	\[
		 \|I_1 u^{\lam}\|_{X^{1}_{1,b}}^2 = \lam^4 \iint_{\m{R}^2} (1+|\xi|)^2 (1+|\tau-\xi^3|)^{2b} m_1^2(\xi) |\wh{u}(\lam\xi, \lam^3\tau)|^2 \,d\xi \,d\tau.
	\]
	By the change of variable $\eta:= \lam\xi$ and $s := \lam^3\tau$, we have 
	\[
		\|I_1 u^{\lam}\|_{X^{1}_{1,b}}^2 = \iint_{\m{R}^2} (1+ \lam^{-1}|\eta|)^2 (1 + \lam^{-3} |s-\eta^3|)^{2b} m_1^2(\lam^{-1}\eta) |\wh{u}(\eta, s)|^2 \,d\eta \,ds.
	\]
	Since $\lam\geq 1$ and $m_1(x)$ is a decreasing function in the radius $|x|$, then $m_1(\lam^{-1}\eta) \geq m_1(\eta)$. Consequently, 
	\[\begin{split}
		\|I_1 u^{\lam}\|_{X^{1}_{1,b}}^2 &\geq \lam^{-2-6b} \iint_{\m{R}^2} (1+ |\eta|)^2 (1 + |s-\eta^3|)^{2b} m_1^2(\eta) |\wh{u}(\eta, s)|^2 \,d\eta \,ds \\
		&= \lam^{-2-6b} \|I_1 u\|_{X^{1}_{1,b}}^2 .
	\end{split}\]
	Similarly, we can build an analogous estimate between $\|I_2 v^{\lam}\|_{X^{4}_{1,b}}$ and $\|I_2 v\|_{X^{4}_{1,b}}$. Hence, 	
	\be\label{scale_fr_conv}\begin{split}
		& \|I_1 u\|_{X^{1}_{1,b} (\m{R} \times [0,\lam^{-3}])} + \|I_2 v\|_{X^{4}_{1,b} (\m{R} \times [0,\lam^{-3}])} \\
		\leq\;\; & \lam^{1+3b} \big( \|I_1 u^{\lam}\|_{X^{1}_{1,b} (\m{R} \times [0,1])} + \|I_2 v^{\lam}\|_{X^{4}_{1,b} (\m{R} \times [0,1])} \big).
	\end{split}\ee
	Combining (\ref{scale_fr_conv}) with (\ref{scaled_fr_est}) and (\ref{scaled_norm}) yields
	\[
	\|I_1 u\|_{X^{1}_{1,b} (\m{R} \times [0,\lam^{-3}])} + \|I_2 v\|_{ X^{4}_{1,b} (\m{R} \times [0,\lam^{-3}])} \leq C \lam^{3b + \frac12 - s} F_0.
	\]
	Due to the choice (\ref{choice_lam}) for $\lam$, we find 
	\[
	\|I_1 u\|_{X^{1}_{1,b} (\m{R} \times [0,\lam^{-3}])} + \|I_2 v\|_{X^{4}_{1,b} (\m{R} \times [0,\lam^{-3}])} \leq C F_{0}^{\frac{6b + 2}{2s+1}}.
	\]
	Finally, the lifespan of $(u,v)$ has the lower bound
	$T_{*} := \lam^{-3} = \delta F_{0}^{-\frac{6}{2s+1}}$, where $\delta$ is a constant that only depends on $s$ and $b$. 
\end{proof}

The most common application of Proposition \ref{Prop, ls1} is when the size $F_0$ of the initial data $(I_1 u_0, I_2v_0)$ is smaller than $1$. Under this assumption, the lifespan can be extended by at least a constant size $\delta$ and the upper bound $F_{0}^{\frac{ 6b+2}{2s+1}}$ in (\ref{ub_ls1}) is simply dominated by $F_{0}$ since $\frac{6b+2}{2s+1} > 1$ and $F_0\leq 1$. Thus, we obtain the following corollary.

\begin{corollary}\label{Cor, ls1}
	Let $\frac34\leq s < 1$ and $\frac12 < b \leq 1$. If $F_0 := \| (I_1 u_0, I_2 v_0) \|_{\mathcal{H}^{1}(\m{R})} \leq 1$, then the lifespan of the local solution $(u,v)$ has a lower bound $\delta$ which is a positive constant only depending on $s$ and $b$. Meanwhile, 
	\[
	\|I_1 u\|_{X^{1}_{1,b} (\m{R} \times [0, \delta])} + \|I_2 v\|_{ X^{4}_{1,b}(\m{R} \times [0,\delta])} \leq C F_{0},
	\]
	where $C$ is a positive constant only depending on $s$ and $b$.
\end{corollary}

\subsection{Key growth estimate of the second modified energy}
Based on Proposition \ref{Prop, ls1}, once the data $(u,v)(\cdot, T)\in \mathcal{H}^{s}(\m{R})$ \big(equivalently $(I_1 u, I_2 v)(\cdot, T)\in \mathcal{H}^{1}(\m{R})$\big) at some point $T$, then by regarding $(u,v)(\cdot, T)$ as an initial data, the local solution can be extended by an amount of time $T_{*}$ such that $I_1 u\in X^{1}_{1,b} (\m{R} \times [T, T+T_{*}])$ and $I_2 v\in X^{4}_{1,b} (\m{R} \times [T, T+T_{*}])$. In order to extend this solution further to an arbitrary long time, it is crucial to control the growth of energies. In the next result, we will establish two quadri-linear estimates $\Lambda_4 (M_{41}; v)$ and $\Lambda_4 (M_{42}; u, u, v, v)$ in space-time domain such that the growth of $E^{(2)}(u,v)$ can be controlled due to formula (\ref{E2_td}).

\begin{lemma}\label{Lemma, est_M4op}
Let $\frac{3}{4} \leq s < 1$, $\frac{1}{2} < b \leq 1$, and $\beta < 1$. If there exists a time interval $[T, T+\delta]$ such that $I_1 u \in X^{1}_{1,b} (\m{R} \times [T, T+\delta])$ and $I_2 v \in X^{4}_{1,b} (\m{R} \times [T, T+\delta])$, then 
\begin{align}
	\left| \int_{T}^{T+\delta} \Lambda_4(M_{41}; v) \dd t \right| &\leq C \delta N^{-\beta} \| I_2 v \|_{X^{4}_{1,b} (\m{R} \times [T, T+\delta])}^4, \label{est_M41} \\
	\left| \int_{T}^{T+\delta} \Lambda_4(M_{42}; u, u, v, v) \dd t \right| &\leq C \delta N^{-\beta} \| I_1 u \|_{X^{1}_{1,b} (\m{R} \times [T, T+\delta])}^2 \| I_2 v \|_{X^{4}_{1,b} (\m{R} \times [T, T+\delta])}^2, \label{est_M42}
\end{align}
where $C$ is a positive constant only depending on $s$, $b$ and $\b$.
\end{lemma}

\begin{proof}
The proofs for (\ref{est_M41}) and (\ref{est_M42}) are very similar, so we will only verify (\ref{est_M42}) below. For ease of notations, we will denote $C$ by a constant that only depends on $s$, $b$ and $\b$, and the values of $C$ may be different from line to line. Moreover, we will simply use $X^{1}_{1,b}$ and $X^{4}_{1,b}$ to represent the spaces $X^{1}_{1,b} (\m{R} \times [T, T+\delta])$ and $X^{4}_{1,b} (\m{R} \times [T, T+\delta])$ respectively.
According to the definition of the operator $\Lam_{4}(M_{42}; u,u,v,v)$, it is equivalent to prove
\be\label{M42op_est_1}\begin{split}
	& \int_{T}^{T+\delta} \int_{\Gamma_4} M_{42}(\xi_1, \xi_2, \xi_3, \xi_4) \, \wh{u}(\xi_1, t) \wh{u}(\xi_2, t) \wh{v}(\xi_3, t) \wh{v}(\xi_4, t) \, d\sigma \, dt \\
	\leq\; & C \delta N^{-\b} \| I_1 u \|_{X^{1}_{1,b}}^2 \| I_2 v \|_{X^{4}_{1,b}}^2,
\end{split}\ee
where $\wh{u}(\xi,t) := \F_{x} u(\xi,t)$ and $\wh{v}(\xi,t) := \F_{x} v(\xi,t)$. 

In the region where $\max\limits_{1\leq j\leq 4} |\xi_j| \leq \frac{N}{4}$, it follows from (\ref{mpl_4}) that $M_{42}(\xi_1, \xi_2, \xi_3, \xi_4) = 2(\xi_2+\xi_3)$. Thus, by using symmetry, 
\[\begin{split}
	&\int_{\Gamma_4 \cap \{ \max\limits_{1\leq j\leq 4} |\xi_j| \leq \frac{N}{4} \}} 2(\xi_2+\xi_3) \, \wh{u}(\xi_1, t) \wh{u}(\xi_2, t) \wh{v}(\xi_3, t) \wh{v}(\xi_4, t) \, d\sigma \\
	=\,\, & \int_{\Gamma_4 \cap \{ \max\limits_{1\leq j\leq 4} |\xi_j| \leq \frac{N}{4} \}} [(\xi_2+\xi_3) + (\xi_1+\xi_4)] \, \wh{u}(\xi_1, t) \wh{u}(\xi_2, t) \wh{v}(\xi_3, t) \wh{v}(\xi_4, t) \, d\sigma = 0,
\end{split}\]
where the last inequality is due to $\sum\limits_{1\leq j\leq 4} \xi_j = 0$ on $\Gamma_4$.
So the integral domain $\Gamma_4$ in (\ref{M42op_est_1}) can be restricted to $\wt{\Gamma}_{4}$, where 
\[
	\wt{\Gamma}_{4} := \Big\{ (\xi_1,\xi_2,\xi_3,\xi_4)\in \Gamma_4: \max_{1\leq j\leq 4} |\xi_j| \geq \frac{N}{4} \Big\}.
\]
Consequently, it suffices to show 
\be\label{M42op_est_2}\begin{split}
	& \int_{T}^{T+\delta} \int_{\wt{\Gamma}_4} |M_{42}(\xi_1, \xi_2, \xi_3, \xi_4)| \, |\wh{u}(\xi_1, t) \wh{u}(\xi_2, t) \wh{v}(\xi_3, t) \wh{v}(\xi_4, t)| \, d\sigma \, dt \\
	\leq\; & C \delta N^{-\b} \| I_1 u \|_{X^{1}_{1,b}}^2 \| I_2 v \|_{X^{4}_{1,b}}^2.
\end{split}\ee
%Without loss of generality, we may assume both $\wh{u}$ and $\wh{v}$ are nonnegative \big(otherwise, one can work with $\wt{u} := \F_{x}^{-1} (|\wh{u}|)$ and $\wt{v} := \F_{x}^{-1} (|\wh{v}|)$\big). 

Denote $\{\Omega_{k}\}_{k\geq 0}$ to be a sequence of sets in the frequency space such that 
\[
	\Omega_0 = \{\xi\in\m{R}: |\xi| \leq 1\}, \quad \Omega_k = \{\xi\in\m{R}: 2^{k-1}\leq |\xi| < 2^{k} \}, \quad\forall\, k\geq 1.
\]
Then we decompose the integral domain $\wt{\Gamma}_{4}$ in (\ref{M42op_est_2}) to be 
$\wt{\Gamma}_{4} = \bigcup\limits_{(k_1, k_2, k_3, k_4)\in \m{N}^{4}} \Omega_{k_1 k_2 k_3 k_4}$,
where 
\[
	\Omega_{k_1 k_2 k_3 k_4} := \{ (\xi_1,\xi_2,\xi_3,\xi_4)\in \wt{\Gamma}_4: \xi_j \in \Omega_{k_j}, j=1,2,3,4 \}, \quad\forall\, (k_1, k_2, k_3, k_4)\in \m{N}^{4}.
\]
Hence, in order to justify (\ref{M42op_est_2}), it reduces to find some constant $\eps > 0$ such that for any $(k_1, k_2, k_3, k_4)\in \m{N}^{4}$, 
\be\label{M42op_est_3}\begin{split}
	& \int_{T}^{T+\delta} \int_{\Omega_{k_1 k_2 k_3 k_4}} |M_{42}(\xi_1, \xi_2, \xi_3, \xi_4)| \, |\wh{u}(\xi_1, t) \wh{u}(\xi_2, t) \wh{v}(\xi_3, t) \wh{v}(\xi_4, t)| \, d\sigma \, dt \\
	\leq\; & C \delta N^{-\b} \Big( \prod_{j=1}^{4} 2^{-k_j \eps} \Big) \| I_1 u \|_{X^{1}_{1,b}}^2 \| I_2 v \|_{X^{4}_{1,b}}^2.
\end{split}\ee

Thanks to Lemma \ref{Lemma, est on M4}, $|M_{42}(\xi_1, \xi_2, \xi_3, \xi_4)| $ is bounded by $N_{max}^{2s-1} N^{2-2s}$, where 
\[N_{max} := N + \max\limits_{1\leq j\leq 4} |\xi_j| \sim \max\limits_{1\leq j\leq 4} |\xi_j| \quad \text{on} \quad \wt{\Gamma}_4.\]
Thus, (\ref{M42op_est_3}) boils down to 
\be\label{M42op_est_4}\begin{split}
	& \int_{T}^{T+\delta} \int_{\Omega_{k_1 k_2 k_3 k_4}} \wt{N}_{max}^{2s-1} N^{2-2s} \, |\wh{u}(\xi_1, t) \wh{u}(\xi_2, t) \wh{v}(\xi_3, t) \wh{v}(\xi_4, t)| \, d\sigma \, dt \\
	\leq\; & C \delta N^{-\b} \Big( \prod_{j=1}^{4} 2^{-k_j \eps} \Big) \| I_1 u \|_{X^{1}_{1,b}}^2 \| I_2 v \|_{X^{4}_{1,b}}^2,
\end{split}\ee
where $\wt{N}_{max} := \max\limits_{1\leq j\leq 4} |\xi_j|$.
Without loss of generality, we may assume $\wt{N}_{\max} = |\xi_1|$ since other situations can be handled similarly. Under this assumption, $k_1 = \max\limits_{1\leq j\leq 4} k_j$, so (\ref{M42op_est_4}) reduces to 
\be\label{M42op_est_5}\begin{split}
 	& \int_{T}^{T+\delta} \int_{\wt{\Omega}_{k_1 k_2 k_3 k_4}} |\xi_1|^{2s-1} N^{2-2s} \, |\wh{u}(\xi_1, t) \wh{u}(\xi_2, t) \wh{v}(\xi_3, t) \wh{v}(\xi_4, t)| \, d\sigma \, dt \\
 	\leq\; & C \delta N^{-\b} 2^{- 4k_1 \eps} \| I_1 u \|_{X^{1}_{1,b}}^2 \| I_2 v \|_{X^{4}_{1,b}}^2,
\end{split}\ee
where 
\[
	\wt{\Omega}_{k_1 k_2 k_3 k_4} := \Big\{ (\xi_1,\xi_2,\xi_3,\xi_4) \in \Omega_{k_1 k_2 k_3 k_4}: |\xi_1| = \max_{1\leq j\leq 4} |\xi_j| \geq \frac{N}{4} \Big\}.
\]
In the above region, $2^{k_1-1}\leq |\xi_1| < 2^{k_1}$ with $k_1\geq 2$, so $2^{-4k_1\eps} \sim |\xi_1|^{-4\eps}$. As a result, (\ref{M42op_est_5}) is equivalent to 
\be\label{M42op_est_6}\begin{split}
	& \int_{T}^{T+\delta} \int_{\wt{\Omega}_{k_1 k_2 k_3 k_4}} |\xi_1|^{2s-1+4\eps} N^{2-2s+\b} \, |\wh{u}(\xi_1, t) \wh{u}(\xi_2, t) \wh{v}(\xi_3, t) \wh{v}(\xi_4, t)| \, d\sigma \, dt \\
	\leq\; & C \delta \| I_1 u \|_{X^{1}_{1,b}}^2 \| I_2 v \|_{X^{4}_{1,b}}^2.
\end{split}\ee

Noticing $\| I_1 u \|_{X^{1}_{1,b}} = \| f \|_{X^{1}_{0,b}}$ and $ \| I_2 v \|_{X^{4}_{1,b}} = \| g \|_{X^{4}_{0,b}}$, where $f$ and $g$ are defined as 
\be\label{f and g} 
	f(x,t) = \F_{x}^{-1} \big[ \la \xi \ra m_1(\xi) \wh{u}(\xi,t) \big], \quad g(x,t) = \F_{x}^{-1} \big[ \la \xi \ra m_2(\xi) \wh{v}(\xi, t) \big]. 
\ee
Then (\ref{M42op_est_6}) further reduces to 
\be\label{M42op_est_7}\begin{split}
	& \int_{T}^{T+\delta} \int_{\wt{\Omega}_{k_1 k_2 k_3 k_4}} \frac{|\xi_1|^{2s-1+4\eps} N^{2-2s+\b}}{M_{*}(\xi_1, \xi_2,\xi_3,\xi_4)} \, |\wh{f}(\xi_1, t) \wh{f}(\xi_2, t) \wh{g}(\xi_3, t) \wh{g}(\xi_4, t)| \, d\sigma \, dt \\
	\leq\; & C \delta \| f \|_{X^{1}_{0,b}}^2 \| g \|_{X^{4}_{0,b}}^2,
\end{split}\ee
where $\wh{f}(\xi,t) := \F_{x} f(\xi,t)$, $\wh{g}(\xi,t) := \F_{x} g(\xi,t)$ and 
\[
	M_{*}(\xi_1, \xi_2, \xi_3, \xi_4) := \Big(\prod_{j=1}^{2} \la \xi_j \ra m_{1}(\xi_j)\bigg) \bigg(\prod_{k=3}^{4} \la \xi_k \ra m_{2}(\xi_k)\Big).
\]

Next, by choosing 
\be\label{eps_choice} 
	0 < \eps < \frac{1-\b}{4}, 
\ee
we claim that for any $t\in [T, T+\delta]$, 
\be\label{space_ml_est1}\begin{split}
	& \int_{\wt{\Omega}_{k_1 k_2 k_3 k_4}} Q(\xi_1, \xi_2,\xi_3,\xi_4) \, |\wh{f}(\xi_1, t) \wh{f}(\xi_2, t) \wh{g}(\xi_3, t) \wh{g}(\xi_4, t)| \, d\sigma \\
	\leq\; & C \| f(\cdot, t) \|_{L^2_x(\m{R})}^2 \| g(\cdot, t) \|_{L^2_x(\m{R})}^2,
\end{split}\ee
where 
\[
	Q(\xi_1, \xi_2,\xi_3,\xi_4) := \frac{|\xi_1|^{2s-1+4\eps} N^{2-2s+\b}}{M_{*}(\xi_1, \xi_2,\xi_3,\xi_4)}, \qquad \forall\, (\xi_1, \xi_2,\xi_3,\xi_4) \in \wt{\Omega}_{k_1 k_2 k_3 k_4}.
\]
Since $\xi_1 = - (\xi_2+\xi_3+\xi_4)$, then one of $|\xi_2|$, $|\xi_3|$ and $|\xi_4|$ is at least $|\xi_1|/3$, which makes it possible to decompose $\wt{\Omega}_{k_1 k_2 k_3 k_4}$ into three subdomains: $\wt{\Omega}_{k_1 k_2 k_3 k_4} = \cup_{j=2}^{4} D_j$, where
\[
	D_{j} := \big\{ (\xi_1, \xi_2,\xi_3,\xi_4) \in \wt{\Omega}_{k_1 k_2 k_3 k_4}: |\xi_j| \geq |\xi_1|/3 \big\}, \quad j = 2,3,4.
\]
Then (\ref{space_ml_est1}) reduces to the following estimate: 
\be\label{space_ml_est2}\begin{split}
	& \sum_{j=2}^{4} \int_{D_j} Q(\xi_1, \xi_2,\xi_3,\xi_4) \, |\wh{f}(\xi_1, t) \wh{f}(\xi_2, t) \wh{g}(\xi_3, t) \wh{g}(\xi_4, t)| \, d\sigma \\
	\leq\; & C \| f(\cdot, t) \|_{L^2_x(\m{R})}^2 \| g(\cdot, t) \|_{L^2_x(\m{R})}^2,
\end{split}\ee

We first study the estimate on $D_2$. For ease of notations, we denote 
\[N_i = 2^{k_i}, \quad i=1,2,3,4.\]
So it is readily seen that $\la \xi_i\ra \sim N_i$ and $m(\xi_i) \sim m(N_i)$ for $i=1,2,3,4$. Since $|\xi_1|/3 \leq |\xi_2| \leq |\xi_1|$ on $D_2$, then it follows from Lemma \ref{Lemma, m1p2} that $\la \xi_2\ra m(\xi_2) \sim N_1 m(N_1)$. In addition, we take advantage of Lemma \ref{Lemma, m1est} to find 
\[
	\la \xi_i\ra m(\xi_i) \gs \la \xi_i\ra^{s} \sim N_i^{s}, \quad i=3,4.
\]
As a result, 
\[
	M_{*}(\xi_1, \xi_2,\xi_3,\xi_4) \gs N_1^2 m^2(N_1) N_3^{s} N_4^{s}\sim N_1^{2s} N^{2-2s} N_3^{s} N_4^{s},
\]
which implies that 
\be\label{Q_est}
	Q(\xi_1, \xi_2,\xi_3,\xi_4) \ls \frac{N^{\b}}{N_1^{1-4\eps} N_3^{s} N_4^{s}}, \quad \forall\, (\xi_1, \xi_2,\xi_3,\xi_4) \in D_2.
\ee
So the estimate (\ref{space_ml_est2}) on $D_2$ boils down to 
\be\label{space_ml_est3}
	\int_{D_2} |\wh{f}(\xi_1, t) \wh{f}(\xi_2, t) \wh{g}(\xi_3, t) \wh{g}(\xi_4, t)| \, d\sigma 
	\leq C \frac{N_1^{1-4\eps} N_3^{s} N_4^{s}}{N^{\b}} \| f(\cdot, t) \|_{L^2_x(\m{R})}^2 \| g(\cdot, t) \|_{L^2_x(\m{R})}^2.
\ee
Now we parameterize $D_2$ using $(\xi_2, \xi_3, \xi_4)$ to obtain 
\[
	\text{LHS of (\ref{space_ml_est3})} \leq \int_{\Omega_{k_4}} \int_{\Omega_{k_3}} \int_{\Omega_{k_2}} \big| \wh{f}(-\xi_2-\xi_3-\xi_4, t) \big| \big| \wh{f}(\xi_2,t) \wh{g}(\xi_3, t) \wh{g}(\xi_4, t) \big| \,d\xi_2 \,d\xi_3 \,d\xi_4.
\]
Then by applying Holder's inequality, we find 
\be\label{space_ml_est4}\begin{split}
	& \text{LHS of (\ref{space_ml_est3})} \\
	\leq \,\, & \bigg(\int_{\Omega_{k_4}} \int_{\Omega_{k_3}} \int_{\Omega_{k_2}}\big| \wh{f}(-\xi_2-\xi_3-\xi_4,t) \big|^2 \,d\xi_2 \,d\xi_3 \,d\xi_4 \bigg)^{\frac12} \|f(\cdot, t)\|_{L^2_\xi(\m{R})} \|g(\cdot, t)\|_{L^2_\xi(\m{R})}^2.
\end{split}\ee
By the change of variable, we know 
\[\begin{split}
	\int_{\Omega_{k_4}} \int_{\Omega_{k_3}} \int_{\Omega_{k_2}}\big| \wh{f}(-\xi_2-\xi_3-\xi_4,t) \big|^2 \,d\xi_2 \,d\xi_3 \,d\xi_4 & \leq \int_{\Omega_{k_4}} \int_{\Omega_{k_3}} \bigg(\int_{\m{R}}\big| \wh{f}(\eta,t) \big|^2 \,d\eta \bigg) \,d\xi_3 \,d\xi_4 \\
	& \sim N_3 N_4 \| f(\cdot, t) \|_{L^2_\xi(\m{R})}^2,
\end{split}\]
where the last relation is due to the fact that $|\Omega_{k_i}| \sim N_i$ for $i=3,4$. Plugging the above estimate into (\ref{space_ml_est4}) yields 
\[
	\text{LHS of (\ref{space_ml_est3})}
	\leq N_3^{\frac12} N_4^{\frac12} \|f(\cdot, t)\|_{L^2_\xi(\m{R})}^2 \|g(\cdot, t)\|_{L^2_\xi(\m{R})}^2.
\]
Then we apply the Plancherel identity to obtain 
\[
	\text{LHS of (\ref{space_ml_est3})}
	= N_3^{\frac12} N_4^{\frac12} \|f(\cdot, t)\|_{L^2_x(\m{R})}^2 \|g(\cdot, t)\|_{L^2_x(\m{R})}^2 
	\ls \text{RHS of (\ref{space_ml_est3})},
\]
where the last inequality is due to $s>\frac12$, $N_1 \gs N$ and $1-4\eps > \b$ thanks to the choice of $\eps$ in (\ref{eps_choice}). Hence, (\ref{space_ml_est3}) is proven. 
By an argument similar to the proof of (\ref{space_ml_est3}), the desired estimates of (\ref{space_ml_est2}) on the domains $D_3$ and $D_4$ follow. Hence, (\ref{space_ml_est2}), and equivalently (\ref{space_ml_est1}), are established. 

Putting (\ref{space_ml_est1}) into (\ref{M42op_est_7}) leads to 
\[
	\text{LHS of (\ref{M42op_est_7})} \ls \int_{T}^{T+\delta} \| f(\cdot, t) \|_{L^2_x(\m{R})}^2 \| g(\cdot, t) \|_{L^2_x(\m{R})}^2 \,dt.
\]
Now we apply the Holder's inequality to achieve
\be\label{L4L2_est1}
	\text{LHS of (\ref{M42op_est_7})} \ls \delta \| f(x,t) \|_{L^\infty_t([T, T+\delta]; L^2_{x}(\m{R}))}^2  \| g(x,t) \|_{L^\infty_t([T, T+\delta]; L^2_{x}(\m{R}))}^2.
\ee
Comparing (\ref{L4L2_est1}) with (\ref{M42op_est_7}), it suffices to prove 
\be\label{FRS_est1}
	\| f\|_{L^\infty_t L^2_{x}} \ls \| f \|_{X_{0,b}^{1}} \quad\text{and}\quad \| g \|_{L^\infty_t L^2_{x}} \ls \| g \|_{X_{0,b}^{4}}.
\ee
By Plancherel identity and Minkowski inequality, 
\be\label{FRS_est2}
	\| f(x,t) \|_{L^\infty_t L^2_{x}} = \| \F_{x} f(\xi,t)\|_{L^\infty_t L^2_{\xi}} \leq \| \F_{x} f(\xi,t)\|_{L^2_{\xi} L^\infty_t}.
\ee
For any fixed $\xi$ and $t$, it follows from the property of the inverse Fourier transform in the time variable that $\F_{x} f(\xi,t) = \int_{\m{R}} e^{it\tau} \F_{t}\F_{x} f(\xi,\tau) \,d\tau$. 
Then we apply the Holder's inequality to find
\be\label{FRS_est3}\begin{split}
	|\F_{x} f(\xi,t)| &\leq \int_{\m{R}} \frac{1}{\la\tau-\xi^3\ra^b} \, | \F_{t}\F_{x} f(\xi,\tau) | \la\tau-\xi^3\ra^b \,d\tau \\
	&\leq C_b^{\frac12} \bigg(\int_{\m{R}} | \F_{t}\F_{x} f(\xi,\tau) |^2 \la\tau-\xi^3\ra^{2b} \,d\tau\bigg)^{\frac12},
\end{split}\ee
where $C_b = \int_{\m{R}} \la\tau-\xi^3\ra^{-2b} \,d\tau = \int_{\m{R}} \la\tau\ra^{-2b} \,d\tau$ is a finite number since $b > 1/2$. Since the right hand side of the above estimate is independent of $t$, plugging (\ref{FRS_est3}) into (\ref{FRS_est2}) yields
\[
	 \| f(x,t) \|_{L^\infty_t L^2_{x}} \ls \bigg(\int_{\m{R}} \int_{\m{R}} | \F_{t}\F_{x} f(\xi,\tau) |^2 \la\tau-\xi^3\ra^{2b} \,d\tau \,d\xi\bigg)^{\frac12} = \|f\|_{X_{0,b}^1}.
\]
By an analogous argument, we can also verify that $\| g \|_{L^\infty_t L^2_{x}} \ls \| g \|_{X_{0,b}^{4}}$. Thus, \eqref{FRS_est1} is justified and the proof of (\ref{est_M42}) finishes.

%Now we estimate the right hand side of (\ref{M42op_est_8}). Without loss of generality, we assume both $\wh{f}$ and $\wh{g}$ are non-negative. Since $b > \frac12$ and $L^{4}_{xt}$ is embedded in $X_{0,\frac12}^{1}$ (see Lemma 2.4 with $\rho=1/2$ and $\theta=0$ in \cite{KPV93Duke}), we have 
%\[
%	\big\|  f \big\|_{L^{4}_{xt}} \ls \big\| f \big\|_{X^{1}_{0,\frac12}}, \quad
%	\big\| g \big\|_{L^{4}_{xt}} \ls \big\| g \big\|_{X^{4}_{0,\frac12}}.
%\]
%Then by applying the Holder's inequality, we obtain 
%\[
%	 \| f \|_{L^{4}_{xt}}^{2}  \| g \|_{L^{4}_{xt}}^{2} \geq \int_{T}^{T+\delta} \int_{\m{R}} f^2(x, t) g^2(x, t) \,dx\, dt.
%\]
%According to the Plancherel identity, we have
%\[\begin{split}
%	\int_{\m{R}} f^2(x, t) g^2(x, t) \,dx &= \int_{\m{R}} \wh{f^2}(\xi,t) \wh{g^2}(-\xi,t) \,d\xi \\
%	&= \int_{\m{R}} \bigg(\int_{\m{R}} \wh{f}(\xi_1,t) \wh{f}(\xi-\xi_1,t) \,d\xi_1\bigg) \bigg(\int_{\m{R}} \wh{g}(\xi_3,t) \wh{g}(-\xi-\xi_3,t) \,d\xi_3\bigg) \,d\xi \\
%	&= \int_{\Gamma_4} \wh{f}(\xi_1, t) \wh{f}(\xi_2, t) \wh{g}(\xi_3, t) \wh{g}(\xi_4, t) \,d\sigma,
%\end{split}\]
%where $d\sigma$ refers to the surface integral on the hyperplane $\Gamma_4$. Hence,
%\be\label{rhs_M42op_est}\begin{split}
%	\| f \|_{X^{1}_{0,b}}^2 \| g \|_{X^{4}_{0,b}}^2 
%	\gs  \| f \|_{L^{4}_{xt}}^{2}  \| g \|_{L^{4}_{xt}}^{2} 
%	\geq \int_{T}^{T+\delta} \int_{\Gamma_4} \wh{f}(\xi_1, t) \wh{f}(\xi_2, t) \wh{g}(\xi_3, t) \wh{g}(\xi_4, t) \,d\sigma \, dt,
%\end{split}\ee
%which justifies (\ref{M42op_est_8}).
\end{proof}

Based on the time derivative (\ref{E2_td}) of $E^{(2)}$ and Lemma \ref{Lemma, est_M4op}, we can control the growth of $E^{(2)}$ now.
\begin{proposition}\label{Prop, growth_E2}
Let $\frac{3}{4} \leq s < 1$, $\frac{1}{2} < b \leq 1$, and $\beta < 1$. If there exists a time interval $[T, T+\delta]$ such that $I_1 u \in X^{1}_{1,b} (\m{R} \times [T, T+\delta])$ and $I_2 v \in X^{4}_{1,b} (\m{R} \times [T, T+\delta])$, then 
\be\label{est_E2}\begin{split}
	& \big| E^{(2)}(u,v)(T+\delta) - E^{(2)}(u,v)(T) \big| \\
	\leq\; & C \delta N^{-\beta} \big( \| I_1 u \|_{X^{1}_{1,b} (\m{R} \times [T, T+\delta])} + \| I_2 v \|_{X^{4}_{1,b} (\m{R} \times [T, T+\delta])} \big)^4,
\end{split}\ee
where $C$ is a positive constant only depending on $s$, $b$ and $\b$.
\end{proposition}

\subsection{Difference between the first and the second modified energies}
Once the growth of the second modified energy $E^{(2)}$ is suitably bounded by Proposition \ref{Prop, growth_E2}, the growth of the first modified energy $E^{(1)}$ can also be controlled, provided the difference between $E^{(2)}$ and $E^{(1)}$ is appropriately estimated.

\begin{lemma}\label{Lemma, diff_E1_E2}
Let $\frac34\leq s < 1$ and $(u,v)\in\mathcal{H}^{s}(\m{R})$. Let $E^{(1)}$ and $E^{(2)}$ be as defined in (\ref{mE_1st_freq}) and (\ref{mE_2nd_freq}). Then 
\be\label{diff_mE}
	|E^{(2)}(u,v) - E^{(1)}(u,v)| \leq C \|I_1 u\|_{H^{1}(\m{R})}  \|I_2 v\|_{H^{1}(\m{R})}^2,
\ee
where $C$ is a positive constant only depending on $s$.
\end{lemma}
\begin{proof}
According to  (\ref{mE_1st_freq}) and (\ref{mE_2nd_freq}), 
\[
	|E^{(2)}(u,v) - E^{(1)}(u,v)| = \big| \Lambda_3(m_1(\xi_1) m_2(\xi_2) m_2(\xi_3); u, v, v) - \Lambda_{3} (\sigma_3; u,v,v)\big|,
\]
where 
\[
	\sigma_3(\xi_1, \xi_2, \xi_3) = \frac{\xi_1^3 m_1^2(\xi_1) + 4\xi_2^3 m_2^2(\xi_2) + 4\xi_3^3 m_2^2(\xi_3)}{ \xi_1^3 + 4\xi_2^3 + 4\xi_3^3}, \quad \forall\, (\xi_1,\xi_2,\xi_3)\in \Gamma_3.
\]
So it suffices to prove that 
\be\label{diff_energy_est1}\left\{\begin{array}{rcl}
	\big| \Lambda_3(m_1(\xi_1) m_2(\xi_2) m_2(\xi_3); u, v, v) \big| &\leq & C \|I_1 u\|_{H^{1}(\m{R})}  \|I_2 v\|_{H^{1}(\m{R})}^2, \\
	|\Lambda_{3} (\sigma_3; u,v,v) | &\leq & C \|I_1 u\|_{H^{1}(\m{R})}  \|I_2 v\|_{H^{1}(\m{R})}^2.
\end{array}\right.\ee
Define $f = I_1 u$ and $g = I_2 v$. Then (\ref{diff_energy_est1}) reduces to 
\be\label{diff_energy_est2}\left\{\begin{array}{rcl}
	\big| \Lambda_3(1; f, g, g) \big| &\leq & C \| f \|_{H^{1}(\m{R})}  \| g \|_{H^{1}(\m{R})}^2, \vspace{0.1in}\\
	\Big| \Lambda_{3} \Big(\dfrac{\sigma_3}{m_1(\xi_1) m_2(\xi_2) m_2(\xi_3)}; f, g, g\Big) \Big| &\leq & C \| f \|_{H^{1}(\m{R})}  \| g \|_{H^{1}(\m{R})}^2.
\end{array}\right.\ee
Since Lemma \ref{Lemma, est on sigma_3} provides a constant bound on $|\sigma_3|_{L^{\infty}}$, and both $m_1$ and $m_2$ are bounded by $1$, then the following estimate (\ref{diff_energy_est3}) implies (\ref{diff_energy_est2}).
\be\label{diff_energy_est3}
\int_{\Gamma_3} \frac{|\wh{f}(\xi_1) \wh{g}(\xi_2) \wh{g}(\xi_3) |}{m_1(\xi_1) m_2(\xi_2) m_2(\xi_3)} \,d\sigma \leq C \| f \|_{H^{1}(\m{R})}  \| g \|_{H^{1}(\m{R})}^2.
\ee

Now it remains to verify (\ref{diff_energy_est3}). According to Lemma \ref{Lemma, m1est}, $m_1(\xi_1) \gs \la \xi_1\ra^{s-1}$ and $m_2(\xi_j) \gs \la \xi_j\ra^{s-1}$ for $j=2,3$, so
it suffices to justify 
\be\label{diff_energy_est4}
\int_{\Gamma_3} 
\la\xi_1\ra^{1-s} |\wh{f}(\xi_1)| \la\xi_2\ra^{1-s} |\wh{g}(\xi_2)| \la\xi_3\ra^{1-s} |\wh{g}(\xi_3)|\,d\sigma 
\leq C \| f \|_{H^{1}(\m{R})}  \| g \|_{H^{1}(\m{R})}^2.
\ee
Define 
\be\label{f1g1}
	f_1 = \F_{x}^{-1}\big[ \la\xi\ra^{1-s} |\wh{f}(\xi)| \big], \quad g_1 = \F_{x}^{-1}\big[ \la\xi\ra^{1-s} |\wh{g}(\xi)| \big].
\ee
Then (\ref{diff_energy_est4}) is equivalent to 
\[
	\int_{\Gamma_3} \wh{f_1}(\xi_1) \wh{g_1}(\xi_2) \wh{g_1}(\xi_3) \,d\sigma \leq C \| f \|_{H^{1}(\m{R})}  \| g \|_{H^{1}(\m{R})}^2.
\]
According to the Plancherel identity, we have 
\[\begin{split}
	\int_{\m{R}} f_1(x) g_1^2(x) \,dx = \int_{\m{R}} \wh{f_1}(\xi_1) \wh{g_1^2}(-\xi_1) \,d\xi_1 &= \int_{\m{R}} \wh{f_1}(\xi_1) \bigg(\int_{\m{R}} \wh{g_1}(\xi_2) \wh{g_1}(-\xi_1-\xi_2) \,d\xi_2\bigg) \,d\xi_1 \\
	& = \int_{\Gamma_3} \wh{f_1}(\xi_1) \wh{g_1}(\xi_2) \wh{g_1}(\xi_3) \,d\sigma.
\end{split}\]
So it boils down to showing 
\be\label{diff_energy_est5}
	\int_{\m{R}} f_1(x) g_1^2(x) \,dx \leq C \| f \|_{H^{1}(\m{R})}  \| g \|_{H^{1}(\m{R})}^2.
\ee
Using Holder's inequality and Sobolev inequality, we have 
\[
	\int_{\m{R}} f_1(x) g_1^2(x) \,dx \leq \|f_1\|_{L^3(\m{R})} \|g_1\|_{L^3(\m{R})}^2 \leq C  \|f_1\|_{H^{\frac16}(\m{R})} \|g_1\|_{H^{\frac16}(\m{R})}^2.
\]
Noting that
\[\|f_1\|_{H^{\frac16}(\m{R})} = \|f\|_{H^{\frac76-s}(\m{R})} \quad\text{and}\quad 
\|g_1\|_{H^{\frac16}(\m{R})} = \|g\|_{H^{\frac76-s}(\m{R})},\] 
so (\ref{diff_energy_est5}) is verified since $\frac76-s \leq 1$.
\end{proof}

\section{Global well-posedness}
\label{Sec, gwp}
\subsection{Relation between the first modified energy and the $H^1$ norm}
\label{Sec, energy_relation}

The local well-posedness of the system (\ref{MB-4}) has been justified and the local solution is guaranteed to exist for a certain amount of time in Proposition \ref{Prop, ls1} for any initial data $(u_0, v_0)\in\mathcal{H}^{s}(\m{R})$. So it suffices to extend the local solution to an arbitrary long time interval. The conserved energy $E(u,v)$ and the first modified energy $E^{(1)}(u,v)$, as defined in (\ref{eq_E}) and (\ref{mE_1st}) respectively, will play crucial role in establishing this goal. We first examine the relation between $E$ and $H^{1}$.

\begin{lemma}\label{Lemma, E and H1}
	Let $f$ and $g$ be any $H^1$ functions on $\R$. Denote $M = \|(f,g)\|_{L^2(\m{R})}$. Then
	\be\label{E and H1}
	\begin{cases}
		\|(f, g)\|_{\mathcal{H}^1(\m{R})}^2 \leq E(f,g) + \frac38 M^{10/3} + M^2, \\
		| E(f,g) | \leq 5 \|(f, g)\|_{\mathcal{H}^1(\m{R})}^2 + \frac34 M^{10/3},
	\end{cases}
	\ee
	where $E(f,g)$ is defined as in \eqref{eq_E}, that is
	$E(f,g) := \int_{\R} \left[ f_x^2(x) + 4 g_x^2(x) - f(x) g^2(x) \right] \,dx$.
\end{lemma}
\begin{proof}
	According to the standard inequality:
	\[
	\|g\|_{L^{\infty}(\m{R})}^2 \leq \| g_{x}\|_{L^2(\m{R})} \|g\|_{L^2(\m{R})}, \quad \forall\, g\in H^{1}(\m{R}),
	\]
	we find
	\[
	\int_{\m{R}} |f| g^2 \,dx \leq \|g\|_{L^{\infty}} \|f\|_{L^2} \|g\|_{L^2} \leq \|g_{x}\|_{L^2}^{1/2} \|f\|_{L^2} \|g\|_{L^2}^{3/2} \leq  \|g_{x}\|_{L^2}^{1/2} M^{5/2}.
	\]
	Thanks to Young's inequality, we deduce that 
	\[
		\|g_{x}\|_{L^2}^{1/2} M^{5/2} \leq \frac14\|g_{x}\|_{L^2}^{2} + \frac34 M^{10/3},
	\]
	which implies that 
	\[
		|E(f,g)| \leq \int_{\R} \left[ f_x^2(x) + 4 g_x^2(x) \right] \,dx + \Big( \frac14 \|g_{x}\|_{L^2}^{2} + \frac34 M^{10/3} \Big) \leq  5 \|(f, g)\|_{\mathcal{H}^1(\m{R})}^2 + \frac34 M^{10/3}.
	\]
	On the other hand, by choosing $\lam = 12^{\frac14}$, it again follows from the Young's inequality that
	\[\begin{split}
			\|g_{x}\|_{L^2}^{1/2} M^{5/2} = \Big(\lam \|g_{x}\|_{L^2}^{1/2}\Big)\Big(\frac{1}{\lam} M^{5/2}\Big) &\leq \frac14 \Big(\lam \|g_{x}\|_{L^2}^{1/2}\Big)^{4} + \frac34 \Big( \frac{1}{\lam} M^{5/2} \Big)^{4/3} \\
			&\leq 3 \|g_{x}\|_{L^2}^{2} + \frac38 M^{10/3}.
	\end{split}\]
	As a result, 
	\[
		E(f,g) \geq \int_{\R} \left[ f_x^2(x) + 4 g_x^2(x) \right] \,dx - \Big( 3 \|g_{x}\|_{L^2}^{2} + \frac38 M^{10/3} \Big) = \|(f_x, g_x)\|_{L^2(\m{R})}^2 - \frac38 M^{10/3}.
	\]
	Hence, (\ref{E and H1}) is justified.
\end{proof}

Since the first modified energy $E^{(1)}$, defined in (\ref{mE_1st}), is closely related to the energy $E$, we can derive a relation between $E^{(1)}$ and the $\mathcal{H}^{1}$ norm of $(I_1 u, I_2 v)$ based on Lemma \ref{Lemma, E and H1}.

\begin{corollary}\label{Cor, E1 and H1}
	Let $\frac34 \leq s < 1$, and let $u$ and $v$ be any $H^{s}$ functions on $\m{R}$. Denote $M = \|(I_1 u, I_2 v)\|_{L^2(\m{R})}$. Then 
	\be\label{E1 and H1}
	\begin{cases}
		\|(I_1 u, I_2 v)\|_{\mathcal{H}^1(\m{R})}^2 \leq E^{(1)}(u,v) + \frac38 M^{10/3} + M^2, \\
		| E^{(1)}(u,v) | \leq 5 \|(I_1 u, I_2 v)\|_{\mathcal{H}^1(\m{R})}^2 + \frac34 M^{10/3},
	\end{cases}
	\ee
	where $E^{(1)}(u,v)$ is as defined in (\ref{mE_1st}), that is $E^{(1)}(u,v) = E(I_1 u, I_2 v)$. 
\end{corollary}
%\begin{proof}
%	Since $u, v \in H^{s}(\m{R})$, then $I_1 u, I_2 v \in H^{1}(\m{R})$ and it follows from (\ref{mE_1st}) that \[E^{(1)}(u,v) = E(I_1 u, I_2 v).\]
%	Denote $M = \|(I_1 u, I_2 v)\|_{L^2(\m{R})}$. Then (\ref{E1 and H1}) follows from (\ref{E and H1}) by setting $f = I_1 u$ and $g = I_2 v$.
%\end{proof}

\subsection{Proof of the main theorem}
After all the preparations in previous sections, we are ready to establish the global well-posedness of the system (\ref{MB-4}).
\begin{proof}[Proof of Theorem \ref{Thm, main}]
In this proof, we fix the parameter $b\in (\frac12, 1)$ in Proposition \ref{Prop, ls} and $\b \in (\frac12, 1)$ in Proposition \ref{Prop, growth_E2}. So the dependence of constants on $b$ or $\b$ will not be tracked anymore. 
Since the initial data $(u_0, v_0) \in \mathcal{H}^{s}(\m{R})$, a local solution of (\ref{MB-4}) is guaranteed to exist for some time according to Proposition \ref{Prop, ls1}, then it suffices to extend this local solution to an arbitrary time interval $[0,T_*]$ in order to prove Theorem \ref{Thm, main}. If the size of $(u_0, v_0)$ is small, then the lifespan of the local solution admits a uniform lower bound as shown in Proposition \ref{Prop, ls} or Corollary \ref{Cor, ls1}. So we first rescale the local solution $(u,v)$ to be $(u^{\lam}, v^{\lam})$ as in (\ref{rescale_soln}) such that the rescaled initial data $(u^{\lam}_0, v^{\lam}_0)$ is small. 
\be\label{rescale_soln}
	(u^{\lam}, v^{\lam})(x,t) = \lam^{-2}(u,v)(\lam^{-1}x, \lam^{-3}t), \quad \forall\, (x,t)\in\m{R}^2, \,\lam\geq 1.
\ee
As discussed in the proof of Proposition \ref{Prop, ls1}, $(u^{\lam}, v^{\lam})$ also satisfies the system (\ref{MB-4}) but with a different initial data $(u^{\lam}_0, v^{\lam}_0)$ as follows:
\be\label{rescale_id}
(u^{\lam}_0, v^{\lam}_0)(x) = (u^{\lam}, v^{\lam})(x,0) = \lam^{-2} (u_0, v_0)(\lam^{-1}x), \quad\forall\, x\in\m{R}.
\ee
In order to prove the lifespan of $(u,v)$ is at least $T_*$, it suffices to show that the lifespan $T_{\lam}$ of $(u^{\lam}, v^{\lam})$ is at least $\lam^3 T_*$.

Denote 
\[
	B_{0} = \|(u_0, v_0)\|_{L^{2}(\m{R})}, \quad B_{s} = \|(u_0, v_0)\|_{\mathcal{H}^{s}(\m{R})}.
\]
Then it follows from (\ref{rescale_soln}) and the conservation of the $L^2$ norm of $(u,v)$ that
\[
\| (u^{\lam}, v^{\lam})(\cdot, t) \|_{L^2(\m{R})} \leq \lam ^{-\frac32} \|(u,v)(\cdot, \lam^{-3} t)\|_{L^2(\m{R})} = \lam ^{-\frac32} B_0, %\quad\forall\, t\geq 0,
\]
which further implies that 
\be\label{bdd_scaled0_I}
\| (I_1 u^{\lam}, I_2 v^{\lam})(\cdot, t) \|_{L^2(\m{R})} \leq \| (u^{\lam}, v^{\lam})(\cdot, t) \|_{L^2(\m{R})} \leq \lam ^{-\frac32} B_0,  %\quad\forall\, t\geq 0.
\ee
Meanwhile, based on (\ref{rescale_id}) and Lemma \ref{Lemma, m1est}, we find
\be\label{bdd_scaled1_I}
\| (I_1 u^{\lam}_0, I_2 v^{\lam}_0)\|_{\mathcal{H}^1(\m{R})} \leq N^{1-s} \| (u^{\lam}_0, v^{\lam}_0)\|_{\mathcal{H}^s(\m{R})} \leq N^{1-s}\lam^{-\frac32} B_s.
\ee
We adopt the notation $\eps_{*}$ as that in Proposition \ref{Prop, ls} such that whenever (\ref{exi_cri}) is satisfied at some time $T$, then the local solution can be extended further to the time $T+1$.
\be\label{exi_cri}
\| (I_1 u^{\lam}, I_2 v^{\lam})(\cdot, T) \|_{\mathcal{H}^{1}(\m{R})} \leq \eps_{*}.
\ee
Without loss of generality, we may assume $\eps_{*}\leq 1$.
For the initial data based on (\ref{bdd_scaled1_I}), we choose a large $\lam$ such that 
\be\label{lam_choice}
	N^{1-s}\lam^{-\frac32} B_s = \eps_0 := \eps_* / C_{*},
\ee
where $C_{*} > 1$ is a large positive number which will be determined later. According to (\ref{lam_choice}),
\be\label{lam_f}
	\lam = \eps_{0}^{-\frac23} N^{\frac23(1-s)} B_s^{\frac23}.
\ee
With this choice of $\lam$, it follows from (\ref{bdd_scaled0_I}) and (\ref{bdd_scaled1_I}) that
\be\label{L2_bdd_I}
\| (I_1 u^{\lam}, I_2 v^{\lam})(\cdot, t) \|_{L^2(\m{R})} \leq \| (u^{\lam}, v^{\lam})(\cdot, t) \|_{L^2(\m{R})} \leq \eps_0 N^{-(1-s)},  %\quad\forall\, t\geq 0, 
\ee
and 
\be\label{H1_bdd_I}
\| (I_1 u^{\lam}_0, I_2 v^{\lam}_0)\|_{\mathcal{H}^1(\m{R})} \leq \eps_0 < \eps_*.
\ee

Thanks to (\ref{H1_bdd_I}) and the criterion (\ref{exi_cri}), the lifespan of the local solution $(u^{\lam}, v^{\lam})$ is at least $1$. For convenience of notation, we denote
\[\begin{cases}
& E^{(1)}(t) := E^{(1)}\big(u^{\lam}(t), v^{\lam}(t) \big), \quad E^{(2)}(t) := E^{(2)}\big(u^{\lam}(t), v^{\lam}(t) \big), \\
& F(t) := \|(I_1 u^{\lam}, I_2 v^{\lam})(\cdot, t)\|_{\mathcal{H}^1(\m{R})}.
\end{cases}\]
Then (\ref{H1_bdd_I}) becomes $F(0)\leq \eps_0 < \eps_*$. In addition, according to Corollary \ref{Cor, E1 and H1} and (\ref{L2_bdd_I}), we obtain 
%for any $t\in[0,1]$,
\be\label{H1_ls_fmE}
	F^{2}(t) \leq E^{(1)}(t) + 2\eps_0^{2} N^{-2(1-s)} \leq E^{(1)}(t) + 2\eps_0^{2},
\ee
and 
\be\label{fmE_ls_H1}
	| E^{(1)}(t) | \leq 5 F^{2}(t) + \eps_0^{10/3} N^{-10(1-s)/3} \leq 5 F^{2}(t) + \eps_0^{2}.
\ee
Plugging $t=0$ in (\ref{fmE_ls_H1}) and combining with (\ref{H1_bdd_I}) yields 
\be\label{E1_fi_bdd}
	| E^{(1)}(0) | \leq 5 \eps_0^2 +  \eps_0^{2} = 6\eps_0^2.
\ee

Next, we will first assume an induction result, Claim \ref{Claim, ind_key}, to finish the current proof, and then turn back to justify Claim \ref{Claim, ind_key}. According to Claim \ref{Claim, ind_key}, by choosing $C_{*}$ to be larger than a constant which only depends on $s$, the lifespan of $(u^{\lam}, v^{\lam})$ is at least $N^{\b}$. Our goal is to verify $N^{\b} \geq \lam^{3} T_{*}$, where $T_{*}$ is a fixed given number. Due to the choice of $\lam$ in (\ref{lam_f}), that is to show 
\[
	N^{\b} \geq \Big(\eps_{0}^{-\frac23} N^{\frac23(1-s)} B_s^{\frac23}\Big)^{3} T_{*},
\]
which is equivalent to 
\be\label{N_large}
	N^{\b - 2(1-s)} \geq \eps_{0}^{-2} B_s^{2} T_{*} = \eps_{*}^{-2} C_{*}^2 B_s^{2} T_{*}.
\ee
Since $\b > \frac12$ and $s\geq \frac34$, then $\b > 2(1-s)$. Therefore, (\ref{N_large}) can be achieved by choosing $N$ large enough. Finally, it remains to establish the following result. 
\begin{claim}\label{Claim, ind_key}
	There exist constants $B_{1}$ and $B_2$, which only depend on $s$, such that when $C_{*} \geq B_{1}$, the lifespan $T_{\lam}$ of $(u^{\lam}, v^{\lam})$ is at least $N^{\b}$. Moreover, for any integer $K \in [0, N^{\b}]$, we have 
	\be\label{bdd_fmEini}
		| E^{(1)}(K) | \leq 8 \eps_0^2, 
	\ee
	and 
	\be\label{bdd_fmEmid}
		|E^{(1)}(t)| \leq |E^{(1)}(0)| + B_2 \eps_0^3  + B_2 (K+1) N^{-\b} \eps_0^4, \quad \forall\, t\in[K, K+1].
	\ee
\end{claim}
\begin{proof}[Proof of Claim \ref{Claim, ind_key}]
We will use induction to justify this claim.

{\textbf Step 1.} We first discuss the case when $K=0$. In this case, it follows from (\ref{E1_fi_bdd}) that $|E^{(1)}(0)| \leq 6\eps_0^2$, so (\ref{bdd_fmEini}) is satisfied. In order to verify (\ref{bdd_fmEmid}), we would like to quantify how fast $E^{(1)}(t)$ can increase on $[0,1]$. By triangle inequality, for any $t\in [0,1]$,
\be\label{tri_ineq}
	| E^{(1)}(t) - E^{(1)}(0) | \leq | E^{(1)}(t) - E^{(2)}(t) | + | E^{(2)}(t) - E^{(2)}(0) | + | E^{(2)}(0) - E^{(1)}(0) |.
\ee
Next, we will estimate the three terms on the right hand side of (\ref{tri_ineq}). For the sake of tracking constants more precisely, we fix the constant notation $C$ in Proposition \ref{Prop, ls}, Proposition \ref{Prop, growth_E2} and Lemma \ref{Lemma, diff_E1_E2} to be $C_1$, $C_2$ and $C_3$ respectively. Since $b \in (\frac12, 1)$ and $\b \in (\frac12, 1)$ have already been fixed, the above mentioned constants only depend on $s$. 

Based on these notations, it follows from Proposition \ref{Prop, ls} that 
\[
	\|I_1 u\|_{X^{1}_{1,b} (\m{R} \times [0, 1])} + \|I_2 v\|_{X^{4}_{1,b} (\m{R} \times [0, 1])} \leq 
	C_1 F(0) \leq C_1 \eps_0,
\]
where the last inequality is due to (\ref{H1_bdd_I}). Then by applying Proposition \ref{Prop, growth_E2} with $T=0$ and $\delta = t \in (0,1]$ yields 
\[
	|E^{(2)}(t) - E^{(2)}(0)| \leq C_2 N^{-\b} \big( \|I_1 u\|_{X^{1}_{1,b} (\m{R} \times [0, 1])} + \|I_2 v\|_{X^{4}_{1,b} (\m{R} \times [0, 1])}  \big)^{4}.
\]
Combining the above two estimates gives 
\be\label{E2ge1}
	|E^{(2)}(t) - E^{(2)}(0)| \leq C_1^4 C_2 N^{-\b} \eps_0^4.
\ee
On the other hand, thanks to Lemma \ref{Lemma, diff_E1_E2}, we find 
\be\label{diff_E12_init}
	|E^{(2)}(0) - E^{(1)}(0)| \leq C_3 \|I_1 u^{\lam}_0\|_{H^1(\m{R})} \|I_2 v^{\lam}_0\|_{H^1(\m{R})}^2 \leq C_3 F^3(0) \leq C_3 \eps_0^3,
\ee
and 
\[
	|E^{(2)}(t) - E^{(1)}(t)| \leq C_3 \|I_1 u^{\lam}(\cdot, t)\|_{H^1(\m{R})} \|I_2 v^{\lam}(\cdot, t)\|_{H^1(\m{R})}^2 \leq C_3 F^{3}(t).
\]
According to (\ref{H1_ls_fmE}), 
\[ 
	F^{3}(t) \leq \big( E^{(1)}(t) + 2\eps_0^2 \big)^{3/2} \leq 4 |E^{(1)}(t)|^{3/2} + 8 \eps_0^3.
\]
Therefore, 
\be\label{diff_E12_t}
	|E^{(2)}(t) - E^{(1)}(t)| \leq 4 C_3 |E^{(1)}(t)|^{3/2} + 8 C_3 \eps_0^3.
\ee

Combining (\ref{tri_ineq})--(\ref{diff_E12_t}) leads to 
\[
	| E^{(1)}(t) - E^{(1)}(0) | \leq C_1^4 C_2 N^{-\b} \eps_0^4 + 9 C_3 \eps_0^3 + 4 C_3 |E^{(1)}(t)|^{3/2},
\]
which implies 
\be\label{E1bdd_init}
	|E^{(1)}(t)| \leq |E^{(1)}(0)| + C_4\big( N^{-\b} \eps_0^4 + \eps_0^3 + |E^{(1)}(t)|^{3/2} \big), \quad \forall\, t\in[0,1],
\ee
where $C_{4}:=\max\{C_1^4 C_2, 9C_3\}$. Since $|E^{(1)}(0)| \leq 6\eps_0^2$ and $\eps_0 < 1$, $E^{(1)}(t)$ satisfies the following ODE inequality:
\be\label{E1ODE_init}\left\{\begin{array}{l}
|E^{(1)}(t)| \leq 6\eps_0^2 + 2 C_4 \eps_0^3 + C_4 |E^{(1)}(t)|^{3/2},  \quad \forall\, t\in[0,1], \vspace{0.05in}\\
|E^{(1)}(0)| \leq 6\eps_0^2.
\end{array}\right.\ee
Meanwhile, since $I_1 u \in X_{1,b}^{1}(\m{R}\times[0,1]) \subset C([0,1]; H^{1}(\m{R}))$ and $I_2 v \in X_{1,b}^{4}(\m{R}\times[0,1]) \subset C([0,1]; H^{1}(\m{R}))$, then we know $E^{(1)}(t)$ is continuous on $[0,1]$. By taking $\eps_0 \leq \frac{1}{10C_4}$, that is 
\be\label{choice_Cstar_1}
	C_{*} \geq 10 C_4 \eps_*,
\ee
then $6\eps_0^2 + 2 C_4 \eps_0^3 + C_4 |9\eps_0^2|^{3/2} < 9\eps_0^2$,
which combines with (\ref{E1ODE_init}) and the continuity of $E^{(1)}(t)$ yields 
\be\label{E1bdd_init_rough}
	|E^{(1)}(t)| \leq 9 \eps_0^2, \quad \forall\, t\in[0,1].
\ee
Substituting (\ref{E1bdd_init_rough}) to the right hand side of (\ref{E1bdd_init}) leads to
\[
	|E^{(1)}(t)| \leq |E^{(1)}(0)| + 30 C_4 \eps_0^3 + C_4 N^{-\b} \eps_0^4, \quad \forall\, t\in[0,1].
\]
This verifies (\ref{bdd_fmEmid}) for $K=0$ as long as $B_2$ is chosen such that 
\be\label{choice_B2_1}
	B_2 \geq 30 C_4.
\ee

{\textbf Step 2.} Next, as a typical induction argument, by assuming there is some integer $m$ in $[1, N^{\b}-1 ]$ such that $T_{\lam}\geq m$, and both (\ref{bdd_fmEini}) and (\ref{bdd_fmEmid}) hold for all integers $K$ in $[0, m-1]$, we aim to prove that $T_{\lam}\geq m+1$, and both (\ref{bdd_fmEini}) and (\ref{bdd_fmEmid}) hold for $K=m$.

Firstly, since (\ref{bdd_fmEmid}) holds for $K=m-1$, then by putting $t=m$ in (\ref{bdd_fmEmid}) leads to 
\[
	|E^{(1)}(m)| \leq |E^{(1)}(0)| + B_2 \eps_0^3  + B_2 m N^{-\b} \eps_0^4 \leq 6\eps_0^2 + 2B_2 \eps_0^3.
\]
By choosing $\eps_0 \leq \frac{1}{B_2}$, i.e. 
\be\label{choice_Cstar_2}
	C_* \geq B_2 \eps_*,
\ee
we find $|E^{(1)}(m)| \leq 8\eps_0^2$, which verifies (\ref{bdd_fmEini}) for $K=m$. Moreover, it then follows from (\ref{H1_ls_fmE}) that $F^2(m) \leq E^{(1)}(t) + 2\eps_0^2 \leq 10\eps_0^2$. By choosing $\eps_0 < \eps_*/4$, i.e. 
\be\label{choice_Cstar_3}
	C_* > 4,
\ee
then $F(m) \leq 4\eps_0\leq \eps_*$. By regarding $(u^{\lam}, v^{\lam})(\cdot, m)$ as the initial data, it then follows from Proposition \ref{Prop, ls} that the lifespan of $(u^{\lam}, v^{\lam})$ can be extended by at least $1$, i.e. $T_{\lam} \geq m+1$, and 
\be\label{Frn_ind}
	\|I_1 u^{\lam}\|_{X^{1}_{1,b} (\m{R} \times [m, m+1])} + \|I_2 v^{\lam}\|_{X^{4}_{1,b} (\m{R} \times [m, m+1])} \leq C_1 F(m) \leq 4 C_1 \eps_0. 
\ee

It remains to justify (\ref{bdd_fmEmid}) for $K=m$ by estimating $E^{(1)}(t)$ for $t\in[m,m+1]$. Again, by triangle inequality, for any $t\in[m,m+1]$,
\be\label{tri_ineq_2}
	| E^{(1)}(t) - E^{(1)}(0) | \leq | E^{(1)}(t) - E^{(2)}(t) | + | E^{(2)}(t) - E^{(2)}(0) | + | E^{(2)}(0) - E^{(1)}(0) |.
\ee
Similar to the derivation of (\ref{diff_E12_init}) and (\ref{diff_E12_t}), we have 
\begin{align}
	|E^{(2)}(0) - E^{(1)}(0)| &\leq C_3 \eps_0^3,  \label{diff_E12_init_ind} \\
	|E^{(2)}(t) - E^{(1)}(t)| &\leq 4 C_3 |E^{(1)}(t)|^{3/2} + 8 C_3 \eps_0^3. \label{diff_E12_t_ind}
\end{align}

The estimate for the term $|E^{(2)}(t) - E^{(2)}(0)|$ is more complicated. Since $t\in[m,m+1]$, by splitting the interval $[0,t]$ into $m+1$ pieces, we obtain
\be\label{diff_E2_dec}
	|E^{(2)}(t) - E^{(2)}(0)| \leq |E^{(2)}(t) - E^{(2)}(m)| + \sum_{j=0}^{m-1} |E^{(2)}(j+1) - E^{(2)}(j)|.
\ee
Combining Proposition \ref{Prop, growth_E2} ($T=m$ and $\delta = t-m \leq 1$) with (\ref{Frn_ind}) yields
\be\label{diff_E2_tail}\begin{split}
	|E^{(2)}(t) - E^{(2)}(m)| &\leq C_2 N^{-\beta} \big( \| I_1 u^{\lam} \|_{X^{1}_{1,b} (\m{R} \times [m, m+1])} + \| I_2 v^{\lam} \|_{X^{4}_{1,b} (\m{R} \times [m, m+1])} \big)^4, \\
	&\leq 256 C_1^4 C_2 N^{-\b} \eps_0^4.
\end{split}\ee
Next, for any $j=0,1,\cdots, m-1$, it also follows from Proposition \ref{Prop, growth_E2} that 
\[
	|E^{(2)}(j+1) - E^{(2)}(j)| \leq C_2 N^{-\beta} \big( \| I_1 u^{\lam} \|_{X^{1}_{1,b} (\m{R} \times [j, j+1])} + \| I_2 v^{\lam} \|_{X^{4}_{1,b} (\m{R} \times [j, j+1])} \big)^4.
\]
Meanwhile, since $(\ref{bdd_fmEini})$ holds for $K=j$, then $|E^{1}(j)| \leq 8\eps_0^2$, which can be combined with (\ref{H1_ls_fmE}) to deduce $F^2(j) \leq |E^{(1)}(j)| + 2\eps_0^2 \leq 10\eps_0^2$. Thus, $F(j)\leq 4\eps_0 \leq \eps_*$ due to (\ref{choice_Cstar_3}), which enables one to apply Proposition \ref{Prop, ls} by regarding $(u^{\lam}, v^{\lam})(\cdot, j)$ as the initial data to conclude
\[
	\|I_1 u^{\lam}\|_{X^{1}_{1,b} (\m{R} \times [j, j+1])} + \|I_2 v^{\lam}\|_{X^{4}_{1,b} (\m{R} \times [j, j+1])} \leq C_1 F(j) \leq 4C_1\eps_0.                             
\]
Therefore, 
\[
	|E^{(2)}(j+1) - E^{(2)}(j)| \leq C_2 N^{-\beta} (4C_1\eps_0)^4 = 256 C_1^4 C_2 N^{-\b} \eps_0^4.
\]
Noticing this upper bound is independent of $j$, so 
\be\label{diff_E2_mid}
	\sum_{j=0}^{m-1} |E^{(2)}(j+1) - E^{(2)}(j)| \leq m\big(256 C_1^4 C_2 N^{-\b} \eps_0^4\big).
\ee
Plugging (\ref{diff_E2_tail}) and (\ref{diff_E2_mid}) into (\ref{diff_E2_dec}) gives
\be\label{diff_E2}
	|E^{(2)}(t) - E^{(2)}(0)| \leq 256 C_1^4 C_2 (m+1) N^{-\b} \eps_0^4.
\ee

Substituting (\ref{diff_E12_init_ind}), (\ref{diff_E12_t_ind}) and (\ref{diff_E2}) into (\ref{tri_ineq_2}) yields 
\[
	| E^{(1)}(t) - E^{(1)}(0) | \leq 4 C_3 |E^{(1)}(t)|^{3/2} + 9 C_3 \eps_0^3 + 256 C_1^4 C_2 (m+1) N^{-\b} \eps_0^4, 
\]
which implies 
\be\label{E1bdd_iter}
|E^{(1)}(t)| \leq |E^{(1)}(0)| + C_5\big[(m+1) N^{-\b} \eps_0^4 + \eps_0^3 + |E^{(1)}(t)|^{3/2} \big], \quad \forall\, t\in[m,m+1],
\ee
where $C_{5}:=\max\{ 256 C_1^4 C_2, 9C_3\}$. Recalling that 
\[|E^{(1)}(0)| \leq 6\eps_0^2, \quad m\leq N^{\b} - 1, \quad  |E^{(1)}(m)| \leq 8\eps_0^2,\] 
so we obtain the following ODE inequality for $E^{(1)}(t)$ on $[m,m+1]$:
\be\label{E1ODE_iter}\left\{\begin{array}{l}
	|E^{(1)}(t)| \leq 6\eps_0^2 + 2 C_5 \eps_0^3 + C_5 |E^{(1)}(t)|^{3/2},  \quad \forall\, t\in[m,m+1], \vspace{0.05in}\\
	|E^{(1)}(m)| \leq 8\eps_0^2.
\end{array}\right.\ee
By taking $\eps_0 \leq \frac{1}{10C_5}$, that is 
\be\label{choice_Cstar_4}
	C_{*} \geq 10 C_5 \eps_*,
\ee
then $6\eps_0^2 + 2 C_5 \eps_0^3 + C_5 |9\eps_0^2|^{3/2} < 9\eps_0^2$,
which combines with (\ref{E1ODE_iter}) and the continuity of $E^{(1)}(t)$ yields 
\be\label{E1bdd_iter_rough}
|E^{(1)}(t)| \leq 9 \eps_0^2, \quad \forall\, t\in[m,m+1].
\ee
Substituting (\ref{E1bdd_iter_rough}) to the right hand side of (\ref{E1bdd_iter}) leads to
\[
|E^{(1)}(t)| \leq |E^{(1)}(0)| + 30 C_5 \eps_0^3 + C_5 (m+1) N^{-\b} \eps_0^4, \quad \forall\, t\in[m,m+1].
\]
This verifies (\ref{bdd_fmEmid}) for $K=m$ as long as $B_2$ is chosen such that 
\be\label{choice_B2_2}
	B_2 \geq 30 C_5.
\ee

As a summary, according to (\ref{choice_B2_1}) and (\ref{choice_B2_2}), we choose $B_2$ such that 
\[
	B_2 = \max\{ 30 C_4, 30 C_5 \} = 30 C_5.
\]
Then based on (\ref{choice_Cstar_1}), (\ref{choice_Cstar_2}), (\ref{choice_Cstar_3}) and (\ref{choice_Cstar_4}), we choose $C_{*}$ such that 
\[
	C_{*} \geq \max\{ 10 C_4\eps_*, B_2 \eps_*, 4, 10 C_5 \eps_* \} = \max\{ B_2\eps_*, 4 \} := B_1.
\]
Since $C_5$ and $\eps_*$ only depend on $s$, the constants $B_1$ and $B_2$ also only depend on $s$. As a standard induction procedure, the combination of Step 1 and Step 2 justifies Claim \ref{Claim, ind_key}. 
\end{proof}
Once Claim \ref{Claim, ind_key} is established, the proof of Theorem \ref{Thm, main} is also finished.
\end{proof}

\appendix

\section{Bilinear estimates for local solutions}
In this section, we provide bilinear estimates which were used in the proof of Proposition \ref{Prop, ls}. These bilinear estimates involve $I$-operators and can be regarded as generalizations of classical bilinear estimates as discussed in e.g. \cite{Oh09, YZ22a}. 
Moreover, we present this result in a more general setting by allowing $\sigma_1$ to be positive while we only need the case of $\sigma_1 = 0$ in the proof of Proposition \ref{Prop, ls}.

\begin{lemma}\label{Lemma, bl_est_I}
Let $\frac34\leq s < 1$, $\frac12 < b \leq 1$ and $\sigma_1 \leq 1-b$. Suppose $u\in X_{s,b}^{1}(\m{R}^2)$ and $v\in X_{s,b}^{4}(\m{R}^2)$. Then the following bilinear estimates hold.
\begin{eqnarray}
	\| I_1(v v_x)\|_{X_{1, b-1+\sigma_1}^{1}(\m{R}^2)} &\leq& C \|I_2 v\|_{X_{1,b}^{4}(\m{R}^2)}^2, \label{bl_est_I1}\\
	\| I_2 [(u v)_{x}]\|_{X_{1, b-1+\sigma_1}^{4}(\m{R}^2)} &\leq& C \|I_1 u\|_{X_{1,b}^{1}} \|I_2 v\|_{X_{1,b}^{4}(\m{R}^2)}, \label{bl_est_I2}
\end{eqnarray}
where $C$ is a constant only depending on $s$ and $b$. 
\end{lemma}

\begin{proof}
The proofs for (\ref{bl_est_I1}) and (\ref{bl_est_I2}) are similar, so we will only justify (\ref{bl_est_I2}) which is more challenging since $u$ and $v$ are not symmetric in (\ref{bl_est_I2}). 
The argument is quite standard which originated from previous works, see e.g. \cite{KPV96, Tao01, Oh09, YZ22a, CKSTT03}, so we will only sketch the key steps. Firstly, according to the definition of Fourier restriction spaces,
\[
\|I_1 u\|_{X_{1,b}^{1}(\m{R}^2)} = \|f_1\|_{L^2_{\xi\tau}(\m{R}^2)}, \quad \|I_2 v\|_{X_{1,b}^{4}(\m{R}^2)} = \|f_2\|_{L^2_{\xi\tau}(\m{R}^2)},
\]
where 
\be\label{fwu}
	f_1(\xi,\tau) = \la\xi\ra\la\tau - \xi^3\ra^{b} m_1(\xi) \wh{u}(\xi,\tau), \quad f_2(\xi,\tau) = \la\xi\ra \la\tau - 4\xi^3\ra^{b} m_2(\xi) \wh{v}(\xi,\tau).
\ee
Meanwhile, 
\[
\| I_2 [(u v)_{x}] \|_{X_{1, b-1+\sigma_1}^{4}(\m{R}^2)} = \| \la\xi\ra \la\tau - 4\xi^3\ra^{b-1+\sigma_1} m_2(\xi) \xi (\wh{u}*\wh{v})(\xi,\tau) \|_{L^2_{\xi\tau}(\m{R}^2)}, 
\]
where $\wh{u}$ and $\wh{v}$ denote the space-time Fourier transforms of $u$ and $v$ respectively, and $\wh{u}*\wh{v}$ means the space-time convolution between $\wh{u}$ and $\wh{v}$.
Then by duality, 
\[\begin{split}
& \| I_2 [(u v)_{x}]\|_{X_{1, b-1+\sigma_1}^{4}} \\
=\;\;& \sup_{f_3\in L^2_{\xi\tau}\setminus \{0\}} \| \la\xi\ra \la\tau - 4\xi^3\ra^{b-1+\sigma_1} m_2(\xi) \xi (\wh{u}*\wh{v})(\xi,\tau) f_3(-\xi, -\tau)\|_{L^1_{\xi\tau}} / \|f_3\|_{L^2_{\xi\tau}}.
\end{split}\]
Since 
\[
	(\wh{u}*\wh{v})(\xi,\tau) = \int_{\m{R}}\int_{\m{R}} \wh{u}(\xi_1, \tau_1) \wh{v}(\xi-\xi_1, \tau-\tau_1) \,d\xi_1 \,d\tau_1,
\]
then by the change of variables: 
\[
	(\xi_1, \xi_2, \xi_3) := (\xi_1, \xi-\xi_1, -\xi) \in \Gamma_3 \quad\text{and}\quad 
	(\tau_1, \tau_2, \tau_3) := (\tau_1, \tau-\tau_1, -\tau) \in \Gamma_3,
\] 
we find
\be\label{temp1}\begin{split}
	& \| \la\xi\ra \la\tau - 4\xi^3\ra^{b-1+\sigma_1} m_2(\xi) \xi (\wh{u}*\wh{v})(\xi,\tau) f_3(-\xi, -\tau)\|_{L^1_{\xi\tau}} \\
	=\,\, & \iint_{\Gamma_3\times \Gamma_3} \big| \la\xi_3\ra \la\tau_3 - 4\xi_3^3\ra^{b-1+\sigma_1} m_2(\xi_3) \xi_3 \wh{u}(\xi_1,\tau_1) \wh{v}(\xi_2,\tau_2) f_3(\xi_3, \tau_3) \big| \,d\sigma_{\xi} \,d\sigma_\tau,
\end{split}\ee
where $d\sigma_\xi$ and $d\sigma_\tau$ refer to the surface integrals for $(\xi_1,\xi_2,\xi_3)\in\Gamma_3$ and $(\tau_1,\tau_2,\tau_3)\in\Gamma_3$ respectively.
According to (\ref{fwu}), we change $u$ and $v$ to $f_1$ and $f_2$ to obtain
\[
	\text{RHS of (\ref{temp1})} = \iint_{\Gamma_3\times \Gamma_3} \frac{ |\xi_3| \la\xi_3\ra m_2(\xi_3) \prod_{i=1}^{3} |f_{i}(\xi_i,\tau_i)| }{\la\xi_1\ra \la\xi_2\ra m_1(\xi_1) m_2(\xi_2) \la L_1\ra^{b} \la L_2\ra^{b} \la L_3\ra^{1-b-\sigma_1}} \,d\sigma_{\xi} \,d\sigma_\tau,
\]
where 
\[ L_1 = \tau_1 - \xi_1^3, \quad L_2 = \tau_2 - 4\xi_2^3, \quad L_3 = \tau_3 - 4\xi_3^3. \]
Hence, in order to verify (\ref{bl_est_I2}), it suffices to prove that
\be\label{bl_est_I_duality1}
\iint_{\Gamma_3\times \Gamma_3} \frac{ |\xi_3| \la\xi_3\ra m_2(\xi_3) \prod_{i=1}^{3} |f_{i}(\xi_i,\tau_i)| }{\la\xi_1\ra \la\xi_2\ra m_1(\xi_1) m_2(\xi_2) \la L_1\ra^{b} \la L_2\ra^{b} \la L_3\ra^{1-b-\sigma_1}} \,d\sigma_{\xi} \,d\sigma_\tau \leq C \prod_{i=1}^{3} \|f_i\|_{L^2_{\xi\tau}}.
%\quad\forall\, f_1, f_2, f_3\in L^2_{\xi\tau},
\ee
Noting 
\[
	\frac{ |\xi_3| \la\xi_3\ra m_2(\xi_3) }{\la\xi_1\ra \la\xi_2\ra m_1(\xi_1) m_2(\xi_2) } 
	= \frac{ \la\xi_3\ra^{1-s} m_2(\xi_3)}{ [\la\xi_1\ra^{1-s} m_1(\xi_1) ] [\la\xi_2\ra^{1-s} m_2(\xi_2)]} \frac{|\xi_3| \la\xi_3\ra^{s}}{\la\xi_1\ra^{s} \la\xi_2\ra^{s}} 
	\leq C \frac{|\xi_3| \la\xi_3\ra^{s}}{\la\xi_1\ra^{s} \la\xi_2\ra^{s}},
\]
where the last inequality is due to $|\xi_3| \leq 2\max\{|\xi_1|, |\xi_2|\}$ and Lemma \ref{Lemma, m1p2}, it boils down to proving
\[
\iint_{\Gamma_3\times \Gamma_3} \frac{|\xi_3| \la\xi_3\ra^{s} \prod_{i=1}^{3} |f_{i}(\xi_i,\tau_i)| }{\la\xi_1\ra^{s} \la\xi_2\ra^{s} \la L_1\ra^{b} \la L_2\ra^{b} \la L_3\ra^{1-b-\sigma_1}} \,d\sigma_{\xi} \,d\sigma_\tau \leq C \prod_{i=1}^{3} \|f_i\|_{L^2_{\xi\tau}}.
\]
Since $\sum_{i=1}^{3}\xi_i = 0$, then $\la\xi_3\ra \leq \la\xi_1\ra\la\xi_2\ra$ and $\frac{\la\xi_3\ra^{s}}{\la\xi_1\ra^{s}\la\xi_2\ra^{s}}$ is a decreasing function in $s$. For  $\frac34\leq s < 1$, it suffices to consider the case when $s=\frac34$, that is to show 
\be\label{bl_est_I_duality2}
\iint_{\Gamma_3\times \Gamma_3} \frac{ |\xi_3| \la\xi_3\ra^{\frac34} \prod_{i=1}^{3} |f_{i}(\xi_i,\tau_i)| }{\la\xi_1\ra^{\frac34} \la\xi_2\ra^{\frac34} \la L_1\ra^{b} \la L_2\ra^{b} \la L_3\ra^{1-b-\sigma_1}} \,d\sigma_{\xi} \,d\sigma_\tau \leq C \prod_{i=1}^{3} \|f_i\|_{L^2_{\xi\tau}}.
\ee

\begin{itemize}
	\item \textbf{Case 1:} $|\xi_3| \leq 1$. \\
	Let $A_1 = \{(\vec{\xi}, \vec{\tau})\in \Gamma_3\times\Gamma_3: |\xi_3|\leq 1\}$, where $\vec{\xi} := (\xi_1, \xi_2, \xi_3)$ and $\vec{\tau} := (\tau_1, \tau_2, \tau_3)$. Then it reduces to deduce that
	\[
	\iint_{A_1} \frac{\prod_{i=1}^{3} |f_{i}(\xi_i,\tau_i)| }{\la L_1\ra^{b} \la L_2\ra^{b} \la L_3\ra^{1-b-\sigma_1}} \,d\sigma_{\xi} \,d\sigma_\tau \leq C \prod_{i=1}^{3} \|f_i\|_{L^2_{\xi\tau}}.
	\]
	This estimate is readily justified for any $\frac12 < b \leq 1$ and $\sigma_1 \leq 1-b$, we omit the details.
	
	\item \textbf{Case 2:} $|\xi_3| > 1$ and $\min\{ |\xi_1|, |\xi_2| \} \geq \frac18 |\xi_3|$. \\
	Let $A_2$ denote this region. Then $\frac{ |\xi_3| \la\xi_3\ra^{3/4}}{\la\xi_1\ra^{3/4} \la\xi_2\ra^{3/4}} \ls |\xi_3|^{1/4}$ on $A_2$. So it suffices to prove
	\[
	\iint_{A_2} \frac{ |\xi_3|^{\frac14} \prod_{i=1}^{3} |f_{i}(\xi_i,\tau_i)| }{ \la L_1\ra^{b} \la L_2\ra^{b} \la L_3\ra^{1-b-\sigma_1}} \,d\sigma_{\xi} \,d\sigma_\tau \leq C \prod_{i=1}^{3} \|f_i\|_{L^2_{\xi\tau}}.
	\]
	According to the argument in \cite{KPV96, YZ22a}, it boils down to the following estimate.  
	\be\label{bl_est_I_duality3}
		\sup_{|\xi_3|>1, \tau_3\in\m{R}} \frac{|\xi_3|^{1/4}}{\la L_3\ra^{1-b-\sigma_1}} \bigg( \int_{\m{R}} \frac{d\xi_1}{\la L_1 + L_2\ra^{2b}} \bigg)^{1/2} \leq C.
	\ee
	For fixed $\xi_3$ and $\tau_3$, by writing $\xi_2 = -\xi_1 - \xi_3$ and $\tau_2 = -\tau_1 - \tau_3$, $L_1 + L_2$ can be represented as a function in $\xi_1$:
	\[ 
		L_1 + L_2 = 3\xi_1^3 + 12 \xi_3 \xi_1^2 + 12\xi_3^2 \xi_1 + 4\xi_3^3 - \tau_3 := P_{\xi_3, \tau_3}(\xi_1). 
	\]
	Then it follows from Lemma 5.4 in \cite{YZ22a} that 
	\[\begin{split}
		\int_{\m{R}} \frac{d\xi_1}{\la L_1 + L_2\ra^{2b}} & = \int_{\m{R}} \frac{d\xi_1}{\la P_{\xi_3, \tau_3}(\xi_1)\ra^{2b}} \\
		&\leq C \big\la 3(12\xi_3^2) - (12\xi_3)^2 \big\ra^{-\frac14} \leq C|\xi_3|^{-1/2}.
	\end{split}\]
	As a result, (\ref{bl_est_I_duality3}) is verified since $\sigma_1 \leq 1-b$.

	\item \textbf{Case 3:} $|\xi_3| > 1$ and $\min\{ |\xi_1|, |\xi_2| \} < \frac18 |\xi_3|$. \\
	Let $A_3$ denote this region. Noting $\la\xi_3\ra \leq \la\xi_1\ra \la\xi_2\ra$, so it reduces to show that 
	\[
	\iint_{A_3} \frac{ |\xi_3| \prod_{i=1}^{3} |f_{i}(\xi_i,\tau_i)| }{ \la L_1\ra^{b} \la L_2\ra^{b} \la L_3\ra^{1-b-\sigma_1}} \,d\sigma_{\xi} \,d\sigma_\tau \leq C \prod_{i=1}^{3} \|f_i\|_{L^2_{\xi\tau}}.
	\]
	Again based on the argument in \cite{KPV96, YZ22a}, it suffices to justify
	\be\label{bl_est_I_duality4}
	\sup_{|\xi_3|>1, \tau_3\in\m{R}} \frac{|\xi_3|}{\la L_3\ra^{1-b-\sigma_1}} \bigg( \int_{ \min\{ |\xi_1|, |\xi_2| \} < \frac18 |\xi_3| } \frac{d\xi_1}{\la L_1 + L_2\ra^{2b}} \bigg)^{1/2} \leq C.
	\ee
	Since $L_1 + L_2 = P_{\xi_3, \tau_3}(\xi_1)$ and 
	\[ 
		P_{\xi_3,\tau_3}'(\xi_1) = 9\xi_1^2 + 24\xi_3\xi_1 + 12\xi_3^2 = 3(3\xi_1+2\xi_3)(\xi_1+2\xi_3),
	\]
	then it follows from the restriction $\min\{ |\xi_1|, |\xi_2| \} < \frac18 |\xi_3| $ that $| P_{\xi_3,\tau_3}'(\xi_1) | \sim |\xi_3|^2$. Hence,
	\[\begin{split}
		 & \int_{ \min\{ |\xi_1|, |\xi_2| \} < \frac18 |\xi_3| } \frac{d\xi_1}{\la L_1 + L_2\ra^{2b}} \\
		 \leq\;\; & \frac{C}{|\xi_3|^2} \int_{ \min\{ |\xi_1|, |\xi_2| \} < \frac18 |\xi_3| } \frac{|P_{\xi_3,\tau_3}'(\xi_1)|}{\la P_{\xi_3, \tau_3}(\xi_1) \ra^{2b}} \, d\xi_1 \leq C |\xi_3|^{-2}.
	\end{split}\]
	Plugging the above estimate into (\ref{bl_est_I_duality4}) finishes the proof since $\sigma_1 \leq 1-b$.
\end{itemize}
\end{proof}

\section{Time Derivative of the $\Lambda_3$ operator}
\label{Sec, td_gLam3}
This section aims to verify formula (\ref{gLam_3_td}) which is about the time derivative of the $\Lambda_3(M_3; u,v,v)$ operator. The idea behind the computation is well-known, see e.g. \cite{CKSTT03, Oh09-2}. But here we still carry out the detailed justification for the convenience of readers.

By definition (\ref{eq_Lambda}), 
\[
	\Lambda_3(M_3; u,v,v) = \int_{\Gamma_3} M_{3}(\xi_1, \xi_2, \xi_3) \wh{u}(\xi_1) \wh{v}(\xi_2) \wh{v}(\xi_3) \, d\sigma.
\]
So $\p_{t} \Lam_3(M_3; u,v,v) = I_1 + I_2 + I_3$, where 
\[\left\{\begin{array}{l}
	I_1 =  \int_{\Gamma_{3}} M_{3}(\xi_1, \xi_2, \xi_3) \wh{u}_t(\xi_1) \wh{v}(\xi_2) \wh{v}(\xi_3) \, d\sigma,    \vspace{0.1in} \\
	I_2 =  \int_{\Gamma_{3}} M_{3}(\xi_1, \xi_2, \xi_3) \wh{u}(\xi_1) \wh{v}_{t}(\xi_2) \wh{v}(\xi_3) \, d\sigma,   \vspace{0.1in} \\
	I_3 =   \int_{\Gamma_{3}} M_{3}(\xi_1, \xi_2, \xi_3) \wh{u}(\xi_1) \wh{v}(\xi_2) \wh{v}_{t}(\xi_3) \, d\sigma.
\end{array}\right.\]
We first compute $I_1$. According to (\ref{mb-freq}) where the formula for $\wh{u}_{t}$ is provided, we find
\[
	I_1 = \int_{\Gamma_{3}} M_{3}(\xi_1, \xi_2, \xi_3) \bigg(   
	i \xi_1^3 \wh{u}(\xi_1) - \frac{1}{2} i \xi_1 \int_{ \eta_1 + \eta_2 = \xi_1} \wh{v}(\eta_1) \wh{v}(\eta_2) \,d\sigma
	\bigg) \wh{v}(\xi_2) \wh{v}(\xi_3) \, d\sigma.
\]
We split $I_1$ into two parts: $I_1 = J_{11} + J_{12}$, where 
\begin{align}
	J_{11} &= i \int_{\Gamma_{3}} \xi_1^3 M_{3}(\xi_1, \xi_2, \xi_3) \wh{u}(\xi_1) \wh{v}(\xi_2) \wh{v}(\xi_3) \, d\sigma = i \Lambda_{3}(\xi_1^3 M_3; u, v, v), \label{J11}\\
	J_{12} &= -\frac{i}{2} \int_{\Gamma_{3}} M_{3}(\xi_1, \xi_2, \xi_3) 
	\xi_1 \bigg(\int_{ \eta_1 + \eta_2 = \xi_1} \wh{v}(\eta_1) \wh{v}(\eta_2) \,d\sigma
	\bigg)\wh{v}(\xi_2) \wh{v}(\xi_3) \, d\sigma. \notag
\end{align}
For $J_{12}$, noticing that $\eta_1 + \eta_2 + \xi_2 + \xi_3 = \xi_1 + \xi_2 + \xi_3 = 0$, so by applying the change of variable: 
\[(y_1, y_2, y_3, y_4) := (\eta_1, \xi_2, \xi_3, \eta_2)\in\Gamma_{4},\]
we obtain 
\be\label{J12}
	J_{12} = -\frac{i}{2} \int_{\Gamma_4} (y_1+y_4) M_{3}(y_1+y_4, y_2, y_3) \prod_{i=1}^{4} \wh{v}(y_i) \,d\sigma.
\ee

Next, we compute $I_2$. Again, based on (\ref{mb-freq}) where the formula for $\wh{v}_{t}$ is provided, we find
\[
	I_2 =  \int_{\Gamma_{3}} M_{3}(\xi_1, \xi_2, \xi_3) \wh{u}(\xi_1) \bigg( 
	4 i \xi_2^3 \wh{v}(\xi_2) - i \xi_2 \int_{\eta_1 + \eta_2 = \xi_2} \wh{u}(\eta_1) \wh{v}(\eta_2) \,d\sigma 
	\bigg)\wh{v}(\xi_3) \, d\sigma.
\]
We split $I_2$ into two parts: $I_2 = J_{21} + J_{22}$, where 
\begin{align}
	J_{21} &= 4i \int_{\Gamma_{3}} \xi_2^3 M_{3}(\xi_1, \xi_2, \xi_3) \wh{u}(\xi_1) \wh{v}(\xi_2) \wh{v}(\xi_3) \, d\sigma = 4i \Lambda_{3}(\xi_2^3 M_3; u, v, v), \label{J21} \\
	J_{22} &= -i \int_{\Gamma_{3}} M_{3}(\xi_1, \xi_2, \xi_3) \wh{u}(\xi_1) \xi_2 \bigg( \int_{\eta_1 + \eta_2 = \xi_2} \wh{u}(\eta_1) \wh{v}(\eta_2) \,d\sigma 
	\bigg)\wh{v}(\xi_3) \, d\sigma.  \notag
\end{align}
For $J_{22}$, noticing that $\xi_1 + \eta_1 + \eta_2 + \xi_3 = \xi_1 + \xi_2 + \xi_3 = 0$, so by applying the change of variable: 
\[(y_1, y_2, y_3, y_4) := (\xi_1, \eta_1, \eta_2, \xi_3)\in\Gamma_{4},\]
we obtain 
\be\label{J22}
	J_{22} = -i \int_{\Gamma_{4}} (y_2+y_3) M_{3}(y_1, y_2+y_3, y_4) \wh{u}(y_1)   \wh{u}(y_2) \wh{v}(y_3) \wh{v}(y_4) \, d\sigma.
\ee
Similarly to the computation of $I_2$, we have $I_3 = J_{31} + J_{32}$, where 
\be\label{I3}\begin{split}
	J_{31} &= 4i \int_{\Gamma_{3}} \xi_3^3 M_{3}(\xi_1, \xi_2, \xi_3) \wh{u}(\xi_1) \wh{v}(\xi_2) \wh{v}(\xi_3) \, d\sigma = 4i \Lambda_{3}(\xi_3^3 M_3; u, v, v), \\
	J_{32} &= -i \int_{\Gamma_{4}} (y_2+y_3) M_{3}(y_1, y_4, y_2+y_3) \wh{u}(y_1)   \wh{u}(y_2) \wh{v}(y_3) \wh{v}(y_4) \, d\sigma.
\end{split}\ee

Combining (\ref{J11})--(\ref{I3}) together yields 
\begin{align*}
	\p_{t} \Lam_3(M_3; u,v,v) = & \big(J_{11} + J_{21} + J_{31}\big) + \big(J_{12} + J_{22} + J_{32}\big) \\
	=& \ i \Lam_3\big( \big[ \xi_1^3 + 4\xi_2^3 + 4\xi_3^3 \big]  M_3; u,v,v\big) - \frac{1}{2} i \Lam_{4}\big( [\xi_1+\xi_4] M_3(\xi_1+\xi_4, \xi_2, \xi_3); v \big) \\
	& - i \Lam_{4}\big( [\xi_2+\xi_3] M_3(\xi_1, \xi_2+\xi_3, \xi_4); u,u,v,v \big) \\
	& - i \Lam_{4}\big( [\xi_2+\xi_3] M_3(\xi_1, \xi_4, \xi_2+\xi_3); u,u,v,v \big),
\end{align*}
which justifies the formula (\ref{gLam_3_td}).

\section*{Acknowledgments}
X.Yang is supported by National Natural Science Foundation of China (No. 12401299), Natural Science Foundation of Jiangsu Province (No. BK20241260), Scientific Research Center of Applied Mathematics of Jiangsu Province (No. BK20233002).

{\small
\bibliographystyle{plain}
\bibliography{Ref-Dispersive}

@article {AC08,
    AUTHOR = {Alvarez, B. and Carvajal, X.},
     TITLE = {On the local well-posedness for some systems of coupled {K}d{V} equations},
   JOURNAL = {Nonlinear Anal.},
  FJOURNAL = {Nonlinear Analysis. Theory, Methods \& Applications. An
              International Multidisciplinary Journal},
    VOLUME = {69},
      YEAR = {2008},
    NUMBER = {2},
     PAGES = {692--715},
      ISSN = {0362-546X},
   MRCLASS = {35Q53 (35B30)},
  MRNUMBER = {2426282},
MRREVIEWER = {Zhaohui Huo},
       DOI = {10.1016/j.na.2007.06.009},
       URL = {https://doi.org/10.1016/j.na.2007.06.009},
}

@article {ACW96,
    AUTHOR = {Ash, J. M. and Cohen, J. and Wang, G.},
     TITLE = {On strongly interacting internal solitary waves},
   JOURNAL = {J. Fourier Anal. Appl.},
  FJOURNAL = {The Journal of Fourier Analysis and Applications},
    VOLUME = {2},
      YEAR = {1996},
    NUMBER = {5},
     PAGES = {507--517},
      ISSN = {1069-5869},
   MRCLASS = {35Q53 (35Q35 76B15 76V05)},
  MRNUMBER = {1412066},
MRREVIEWER = {F. Pempinelli},
       DOI = {10.1007/s00041-001-4041-4},
       URL = {https://doi.org/10.1007/s00041-001-4041-4},
}

@article {Bou93a,
	AUTHOR = {Bourgain, J.},
	TITLE = {Fourier transform restriction phenomena for certain lattice
	subsets and applications to nonlinear evolution equations.
	{I}. {S}chr\"{o}dinger equations},
	JOURNAL = {Geom. Funct. Anal.},
	FJOURNAL = {Geometric and Functional Analysis},
	VOLUME = {3},
	YEAR = {1993},
	NUMBER = {2},
	PAGES = {107--156},
	ISSN = {1016-443X},
	MRCLASS = {35Q55 (11L07 35B10)},
	MRNUMBER = {1209299},
	MRREVIEWER = {Yun Mei Chen},
	DOI = {10.1007/BF01896020},
	URL = {https://doi.org/10.1007/BF01896020},
}

@article {BPST92,
	AUTHOR = {Bona, J. L. and Ponce, G. and Saut, J.-C. and
	Tom, M. M.},
	TITLE = {A model system for strong interaction between internal
	solitary waves},
	JOURNAL = {Comm. Math. Phys.},
	FJOURNAL = {Communications in Mathematical Physics},
	VOLUME = {143},
	YEAR = {1992},
	NUMBER = {2},
	PAGES = {287--313},
	ISSN = {0010-3616},
	MRCLASS = {35Q35 (76B25 76V05)},
	MRNUMBER = {1145797},
	MRREVIEWER = {Alan Jeffrey},
	URL = {http://projecteuclid.org/euclid.cmp/1104248957},
}

@article {CCT03,
    AUTHOR = {Christ, M. and Colliander, J. and Tao, T.},
     TITLE = {Asymptotics, frequency modulation, and low regularity
              ill-posedness for canonical defocusing equations},
   JOURNAL = {Amer. J. Math.},
  FJOURNAL = {American Journal of Mathematics},
    VOLUME = {125},
      YEAR = {2003},
    NUMBER = {6},
     PAGES = {1235--1293},
      ISSN = {0002-9327},
   MRCLASS = {35Q53 (35Q55 35R25 37K10 37K15)},
  MRNUMBER = {2018661},
MRREVIEWER = {John Albert},
       URL =
              {http://muse.jhu.edu/journals/american_journal_of_mathematics/v125/125.6christ.pdf},
}

@article {CKSTT03,
    AUTHOR = {Colliander, J. and Keel, M. and Staffilani, G. and Takaoka, H. and Tao, T.},
     TITLE = {Sharp global well-posedness for {K}d{V} and modified {K}d{V} on {$\mathbb R$} and {$\mathbb T$}},
   JOURNAL = {J. Amer. Math. Soc.},
  FJOURNAL = {Journal of the American Mathematical Society},
    VOLUME = {16},
      YEAR = {2003},
    NUMBER = {3},
     PAGES = {705--749},
      ISSN = {0894-0347},
   MRCLASS = {35Q53 (35B30 37K10)},
  MRNUMBER = {1969209},
MRREVIEWER = {Vladislav G. Dubrovsky},
       DOI = {10.1090/S0894-0347-03-00421-1},
       URL = {https://doi.org/10.1090/S0894-0347-03-00421-1},
}

@article {Fen94,
    AUTHOR = {Feng, X.},
     TITLE = {Global well-posedness of the initial value problem for the
              {H}irota-{S}atsuma system},
   JOURNAL = {Manuscripta Math.},
  FJOURNAL = {Manuscripta Mathematica},
    VOLUME = {84},
      YEAR = {1994},
    NUMBER = {3-4},
     PAGES = {361--378},
      ISSN = {0025-2611},
   MRCLASS = {35Q99 (35B45)},
  MRNUMBER = {1291126},
MRREVIEWER = {Meng Ru Li},
       DOI = {10.1007/BF02567462},
       URL = {https://doi.org/10.1007/BF02567462},
}

@article {GG84,
    AUTHOR = {Gear, J. A. and Grimshaw, R.},
     TITLE = {Weak and strong interactions between internal solitary waves},
   JOURNAL = {Stud. Appl. Math.},
  FJOURNAL = {Studies in Applied Mathematics},
    VOLUME = {70},
      YEAR = {1984},
    NUMBER = {3},
     PAGES = {235--258},
      ISSN = {0022-2526},
   MRCLASS = {76B25 (35Q20 76V05)},
  MRNUMBER = {742590},
MRREVIEWER = {A. A. Minzoni},
       DOI = {10.1002/sapm1984703235},
       URL = {https://doi.org/10.1002/sapm1984703235},
}

@article {Guo09,
    AUTHOR = {Guo, Z.},
     TITLE = {Global well-posedness of {K}orteweg-de {V}ries equation in
              {$H^{-3/4}(\mathbb R)$}},
   JOURNAL = {J. Math. Pures Appl. (9)},
  FJOURNAL = {Journal de Math\'{e}matiques Pures et Appliqu\'{e}es. Neuvi\`eme S\'{e}rie},
    VOLUME = {91},
      YEAR = {2009},
    NUMBER = {6},
     PAGES = {583--597},
      ISSN = {0021-7824},
   MRCLASS = {35Q53 (35B30)},
  MRNUMBER = {2531556},
MRREVIEWER = {John Albert},
       DOI = {10.1016/j.matpur.2009.01.012},
       URL = {https://doi.org/10.1016/j.matpur.2009.01.012},
}

@article {HS81,
    AUTHOR = {Hirota, R. and Satsuma, J.},
     TITLE = {Soliton solutions of a coupled {K}orteweg-de {V}ries equation},
   JOURNAL = {Phys. Lett. A},
  FJOURNAL = {Physics Letters. A},
    VOLUME = {85},
      YEAR = {1981},
    NUMBER = {8-9},
     PAGES = {407--408},
      ISSN = {0031-9163},
   MRCLASS = {35Q20 (76B25)},
  MRNUMBER = {632382},
       DOI = {10.1016/0375-9601(81)90423-0},
       URL = {https://doi.org/10.1016/0375-9601(81)90423-0},
}

@article {KT06,
    AUTHOR = {Kappeler, T. and Topalov, P.},
     TITLE = {Global wellposedness of {K}d{V} in {$H^{-1}(\mathbb T,\mathbb R)$}},
   JOURNAL = {Duke Math. J.},
  FJOURNAL = {Duke Mathematical Journal},
    VOLUME = {135},
      YEAR = {2006},
    NUMBER = {2},
     PAGES = {327--360},
      ISSN = {0012-7094},
   MRCLASS = {35Q53 (35B30)},
  MRNUMBER = {2267286},
MRREVIEWER = {Pierre A. Vuillermot},
       DOI = {10.1215/S0012-7094-06-13524-X},
       URL = {https://doi.org/10.1215/S0012-7094-06-13524-X},
}

@article {KPV93Duke,
    AUTHOR = {Kenig, C. E. and Ponce, G. and Vega, L.},
     TITLE = {The {C}auchy problem for the {K}orteweg-de {V}ries equation in
              {S}obolev spaces of negative indices},
   JOURNAL = {Duke Math. J.},
  FJOURNAL = {Duke Mathematical Journal},
    VOLUME = {71},
      YEAR = {1993},
    NUMBER = {1},
     PAGES = {1--21},
      ISSN = {0012-7094},
   MRCLASS = {35Q53},
  MRNUMBER = {1230283},
MRREVIEWER = {Guo Cheng Zhu},
       DOI = {10.1215/S0012-7094-93-07101-3},
       URL = {https://doi.org/10.1215/S0012-7094-93-07101-3},
}

@article {KPV96,
    AUTHOR = {Kenig, C. E. and Ponce, G. and Vega, L.},
     TITLE = {A bilinear estimate with applications to the {K}d{V} equation},
   JOURNAL = {J. Amer. Math. Soc.},
  FJOURNAL = {Journal of the American Mathematical Society},
    VOLUME = {9},
      YEAR = {1996},
    NUMBER = {2},
     PAGES = {573--603},
      ISSN = {0894-0347},
   MRCLASS = {35Q53 (35Bxx)},
  MRNUMBER = {1329387},
MRREVIEWER = {F. Pempinelli},
       DOI = {10.1090/S0894-0347-96-00200-7},
       URL = {https://doi.org/10.1090/S0894-0347-96-00200-7},
}

@article {Kis09,
    AUTHOR = {Kishimoto, N.},
     TITLE = {Well-posedness of the {C}auchy problem for the {K}orteweg-de
              {V}ries equation at the critical regularity},
   JOURNAL = {Differential Integral Equations},
  FJOURNAL = {Differential and Integral Equations. An International Journal
              for Theory \& Applications},
    VOLUME = {22},
      YEAR = {2009},
    NUMBER = {5-6},
     PAGES = {447--464},
      ISSN = {0893-4983},
   MRCLASS = {35Q53 (35B30)},
  MRNUMBER = {2501679},
MRREVIEWER = {Liana L. Fleming},
}

@article {KV19,
	AUTHOR = {Killip, R. and Vi\c{s}an, M.},
	TITLE = {Kd{V} is well-posed in {$H^{-1}$}},
	JOURNAL = {Ann. of Math. (2)},
	FJOURNAL = {Annals of Mathematics. Second Series},
	VOLUME = {190},
	YEAR = {2019},
	NUMBER = {1},
	PAGES = {249--305},
	ISSN = {0003-486X},
	MRCLASS = {35Q53 (35B30 37K10)},
	MRNUMBER = {3990604},
	MRREVIEWER = {John Albert},
	DOI = {10.4007/annals.2019.190.1.4},
	URL = {https://doi.org/10.4007/annals.2019.190.1.4},
}

@article {LP04,
    AUTHOR = {Linares, F. and Panthee, M.},
     TITLE = {On the {C}auchy problem for a coupled system of {K}d{V}  equations},
   JOURNAL = {Commun. Pure Appl. Anal.},
  FJOURNAL = {Communications on Pure and Applied Analysis},
    VOLUME = {3},
      YEAR = {2004},
    NUMBER = {3},
     PAGES = {417--431},
      ISSN = {1534-0392},
   MRCLASS = {35Q53 (35Q35)},
  MRNUMBER = {2098292},
MRREVIEWER = {Xian Guo Geng},
       DOI = {10.3934/cpaa.2004.3.417},
       URL = {https://doi.org/10.3934/cpaa.2004.3.417},
}

@article {MB03,
    AUTHOR = {Majda, A. J. and Biello, J. A.},
     TITLE = {The nonlinear interaction of barotropic and equatorial
              baroclinic {R}ossby waves},
   JOURNAL = {J. Atmospheric Sci.},
  FJOURNAL = {Journal of the Atmospheric Sciences},
    VOLUME = {60},
      YEAR = {2003},
    NUMBER = {15},
     PAGES = {1809--1821},
      ISSN = {0022-4928},
   MRCLASS = {86A10 (76B65)},
  MRNUMBER = {1994152},
MRREVIEWER = {Leo R. M. Maas},
       DOI = {10.1175/1520-0469(2003)060<1809:TNIOBA>2.0.CO;2},
       URL =
              {https://doi.org/10.1175/1520-0469(2003)060<1809:TNIOBA>2.0.CO;2},
}

@article {Mol11,
    AUTHOR = {Molinet, L.},
     TITLE = {A note on ill posedness for the {K}d{V} equation},
   JOURNAL = {Differential Integral Equations},
  FJOURNAL = {Differential and Integral Equations. An International Journal
              for Theory \& Applications},
    VOLUME = {24},
      YEAR = {2011},
    NUMBER = {7-8},
     PAGES = {759--765},
      ISSN = {0893-4983},
   MRCLASS = {35Q53 (35B30)},
  MRNUMBER = {2830706},
MRREVIEWER = {Luiz Gustavo Farah},
}

@article {Mol12,
    AUTHOR = {Molinet, L.},
     TITLE = {Sharp ill-posedness results for the {K}d{V} and m{K}d{V} equations on the torus},
   JOURNAL = {Adv. Math.},
  FJOURNAL = {Advances in Mathematics},
    VOLUME = {230},
      YEAR = {2012},
    NUMBER = {4-6},
     PAGES = {1895--1930},
      ISSN = {0001-8708},
   MRCLASS = {35Q53 (35A01 35B30 35B45 35D30 35R25)},
  MRNUMBER = {2927357},
MRREVIEWER = {Corentin Audiard},
       DOI = {10.1016/j.aim.2012.03.026},
       URL = {https://doi.org/10.1016/j.aim.2012.03.026},
}

@article {Oh09,
    AUTHOR = {Oh, T.},
     TITLE = {Diophantine conditions in well-posedness theory of coupled
              {K}d{V}-type systems: local theory},
   JOURNAL = {Int. Math. Res. Not.},
  FJOURNAL = {International Mathematics Research Notices},
      YEAR = {2009},
    NUMBER = {18},
     PAGES = {3516--3556},
      ISSN = {1073-7928},
   MRCLASS = {35Q53 (35B30)},
  MRNUMBER = {2535009},
MRREVIEWER = {Justin A. Holmer},
       DOI = {10.1093/imrn/rnp063},
       URL = {https://doi.org/10.1093/imrn/rnp063},
}

@article {Oh09-2,
    AUTHOR = {Oh, T.},
     TITLE = {Diophantine conditions in global well-posedness for coupled
              {K}d{V}-type systems},
   JOURNAL = {Electron. J. Differential Equations},
  FJOURNAL = {Electronic Journal of Differential Equations},
      YEAR = {2009},
     PAGES = {No. 52, 48},
      ISSN = {1072-6691},
   MRCLASS = {35Q53 (35B10 35B30)},
  MRNUMBER = {2505110},
MRREVIEWER = {Ming Mei},
}

@article {ST00,
    AUTHOR = {Saut, J.-C. and Tzvetkov, N.},
     TITLE = {On a model system for the oblique interaction of internal
              gravity waves},
      NOTE = {Special issue for R. Temam's 60th birthday},
   JOURNAL = {ESAIM Math. Model. Numer. Anal.},
  FJOURNAL = {ESAIM Mathematical Modelling and Numerical Analysis},
    VOLUME = {34},
      YEAR = {2000},
    NUMBER = {2},
     PAGES = {501--523},
      ISSN = {0764-583X},
   MRCLASS = {35Q53 (35A22 35B20)},
  MRNUMBER = {1765672},
MRREVIEWER = {Xian Guo Geng},
       DOI = {10.1051/m2an:2000153},
       URL = {https://doi.org/10.1051/m2an:2000153},
}

@article {Tao01,
    AUTHOR = {Tao, T.},
     TITLE = {Multilinear weighted convolution of {$L^2$}-functions, and
              applications to nonlinear dispersive equations},
   JOURNAL = {Amer. J. Math.},
  FJOURNAL = {American Journal of Mathematics},
    VOLUME = {123},
      YEAR = {2001},
    NUMBER = {5},
     PAGES = {839--908},
      ISSN = {0002-9327},
   MRCLASS = {35Q53 (35L05 35Q55)},
  MRNUMBER = {1854113},
MRREVIEWER = {Thierry Cazenave},
       URL =
              {http://muse.jhu.edu/journals/american_journal_of_mathematics/v123/123.5tao.pdf},
}

@article {YZ22a,
	AUTHOR = {Yang, X. and Zhang, B.-Y.},
	TITLE = {Local well-posedness of the coupled {K}d{V}-{K}d{V} systems on
	{$\mathbb{R}$}},
	JOURNAL = {Evol. Equ. Control Theory},
	FJOURNAL = {Evolution Equations and Control Theory},
	VOLUME = {11},
	YEAR = {2022},
	NUMBER = {5},
	PAGES = {1829--1871},
	ISSN = {2163-2472},
	MRCLASS = {35Q53},
	MRNUMBER = {4475878},
	DOI = {10.3934/eect.2022002},
	URL = {https://doi.org/10.3934/eect.2022002},
}

@article {YZ22b,
	AUTHOR = {Yang, X. and Zhang, B.-Y.},
	TITLE = {Well-posedness and critical index set of the {C}auchy problem
	for the coupled {K}d{V}-{K}d{V} systems on {$\mathbb{T}$}},
	JOURNAL = {Discrete Contin. Dyn. Syst.},
	FJOURNAL = {Discrete and Continuous Dynamical Systems. Series A},
	VOLUME = {42},
	YEAR = {2022},
	NUMBER = {11},
	PAGES = {5167--5199},
	ISSN = {1078-0947},
	MRCLASS = {35Q53 (11J04 11J72 35G55 35L56)},
	MRNUMBER = {4483989},
	DOI = {10.3934/dcds.2022090},
	URL = {https://doi.org/10.3934/dcds.2022090},
}

@article {YLZ24,
	AUTHOR = {Yang, X. and Li, S. and Zhang, B.-Y.},
	TITLE = {Effect of lower order terms on the well-posedness of
	{M}ajda-{B}iello systems},
	JOURNAL = {J. Evol. Equ.},
	FJOURNAL = {Journal of Evolution Equations},
	VOLUME = {24},
	YEAR = {2024},
	NUMBER = {4},
	PAGES = {Paper No. 95, 32},
	ISSN = {1424-3199},
	MRCLASS = {35Q53 (35D30 35G55 35L56)},
	MRNUMBER = {4817717},
	DOI = {10.1007/s00028-024-01022-0},
	URL = {https://doi.org/10.1007/s00028-024-01022-0},
}
}

\bigskip

\thanks{(X. Yang) School of Mathematics, Southeast University, Nanjing, Jiangsu 211189, China} 

\thanks{Email: xinyang@seu.edu.cn}

\end{document}